\documentclass[12pt]{amsart}
\usepackage{amsmath,amssymb,amscd,amsthm,amsbsy,times}

\def\XXint#1#2#3{{\setbox0=\hbox{$#1{#2#3}{\int}$}
     \vcenter{\hbox{$#2#3$}}\kern-.5\wd0}}

\usepackage[initials,short-journals]{amsrefs}
\usepackage[shortlabels,inline]{enumitem}
\DeclareFontFamily{U}{rsfs}{\skewchar\font"7F}
\DeclareFontShape{U}{rsfs}{m}{n}{
	<-6> rsfs5
	<6-8> rsfs7
	<8-> rsfs10
	}{}
\DeclareMathAlphabet{\mathscr}{U}{rsfs}{m}{n}
\usepackage{graphicx}
\usepackage{array}
\usepackage{comment}
\usepackage[all]{xy}
\newcommand{\CF}{{\mathcal F}}
\newcommand{\C}{{\mathbb C}}

\newcommand{\R}{{\mathbb R}}
\newcommand{\Z}{{\mathbb Z}}
\newcommand{\N}{{\mathbb N}}

\newcommand{\GV}{\mathrm{GV}}
\newcommand{\FLK}{\mathrm{FLK}}

\newcommand{\WU}{\mathrm{WU}}
\newcommand{\WO}{\mathrm{WO}}
\newcommand{\W}{\mathrm{W}}

\newcommand{\D}{\mathrm{D}}
\DeclareMathOperator{\tr}{tr}
\newcommand{\gl}{\mathfrak{gl}}
\newcommand{\GL}{\mathrm{GL}}

\newcommand{\id}{{\mathrm{id}}}
\DeclareMathOperator{\Vol}{\mathrm{vol}}
\DeclareMathOperator{\Ad}{\mathrm{Ad}}
\DeclareMathOperator{\ord}{\mathrm{ord}}

\newcommand{\Wedge}{{\textstyle\bigwedge}}
\newcommand{\pdif}[2]{\dfrac{\partial#1}{\partial#2}}

\newcommand{\norm}[1]{\lvert#1\rvert}

\DeclareFontFamily{U}{mathx}{}
\DeclareFontShape{U}{mathx}{m}{n}{<-> mathx10}{}
\DeclareSymbolFont{mathx}{U}{mathx}{m}{n}
\DeclareMathAccent{\widecheck}{0}{mathx}{"71}
 \newtheorem{theorem}{Theorem}[section]
 \newtheorem{corollary}[theorem]{Corollary}
 \newtheorem{proposition}[theorem]{Proposition}
 \newtheorem{lemma}[theorem]{Lemma}

\theoremstyle{definition}
 \newtheorem{definition}[theorem]{Definition}
 
 \newtheorem{example}[theorem]{Example}
\theoremstyle{remark}
 \newtheorem{remark}[theorem]{Remark}
\everymath{\displaystyle}
\title{On deformations of foliations and characteristic classes}
\author{Taro Asuke}
\address{Graduate School of Mathematical Sciences, University of Tokyo, 3-8-1 Komaba, Meguro-ku, Tokyo 153-8914, Japan}
\email{asuke@ms.u-tokyo.ac.jp}
\date{March 22, 2026}
\thanks{This work is partly supported by JSPS KAKENHI Grant Number JP21H00980.}
\allowdisplaybreaks
\begin{document}
\begin{abstract}
We study characteristic classes for deformations of foliations.
Those classes include known classes such as the Godbillon--Vey class and the Fuks--Lodder--Kotschick class.
We introduce a certain differential graded algebra (DGA for short) which recovers the Bott vanishing and some formulae by Heitsch.
Some basic properties and structures of the cohomology of those DGA's are discussed.
In particular, it is shown that at the level of the cohomology of DGA, there are some classes which cannot be described by the Godbillon--Vey class and the Fuks--Lodder--Kotschick class.
It is also shown that if a certain type of characteristic classes admit non-trivial deformations in examples, then they yield another kind of classes which admit also non-trivial deformations.
\end{abstract}
\maketitle
\setlength{\baselineskip}{16pt}

\section*{Introduction}
Let $\CF$ be a real codimension-one foliation which is transversely orientable.
Then, there is a $1$-form $\omega$, such that $T\CF=\ker\omega$.
By Frobenius' theorem, there exists a $1$-form $\eta$ such that $d\omega=\omega\wedge\eta$.
Moreover, there is a $1$-form $\mu$ such that $d\eta=\omega\wedge\mu$.
Then, we have $d\eta^2=0$ which is the Bott vanishing.
The $3$-form $\eta\wedge d\eta$ is closed and represents a class independent of choices.
This class is called the \textit{Godbillon-Vey class}.
Let $\{\CF_s\}$ be a smooth family of real codimension-one, transversely orientable foliations.
Then, we have a smooth families of $1$-forms $\{\omega_s\}$, $\{\eta_s\},\{\mu_s\}$ such that $d\eta_s=\omega_s\wedge\mu_s$.
If `$\dot{\hphantom{m}}$' denotes the derivative at $s=0$, then we have $(d\eta)\dot{\,}=d\dot{\eta}$, $d\dot{\eta}=\dot{\omega}\wedge\mu+\omega\wedge\dot{\mu}$, where $\omega=\omega_0$, $\eta_0=\eta$ and $\mu=\mu_0$.
It follows that $d\eta\wedge d\dot{\eta}=0$ and that $\eta\wedge\dot{\eta}\wedge d\eta$ is closed.
Indeed, the latter $4$-form represents a characteristic class of families of foliations \cite{Fuks}, \cite{Kotschick}, \cite{Lodder} and called the \textit{Fuks--Lodder--Kotschick} class.
It is not difficult to show that the Godbillon--Vey class, its derivative and the Fuks--Lodder--Kotschick class are the only characteristic classes obtained from $\eta$, $\dot\eta$, $d\eta$ and $d\dot\eta$.
We can consider differentials of higher order.
For example, we have $d\ddot{\eta}=\ddot{\omega}\wedge\mu+2\dot{\omega}\wedge\dot{\mu}+\omega\wedge\ddot{\mu}$.
Hence $d\ddot{\eta}\wedge d\eta+(d\dot{\eta})^2=0$, and $\eta\wedge\ddot{\eta}\wedge d\eta+\eta\wedge\dot{\eta}\wedge d\dot\eta$ is closed.
The latter form represents also a characteristic class for families of foliations.
Actually, we have $\eta\wedge\ddot{\eta}\wedge d\eta+\eta\wedge\dot{\eta}\wedge d\dot\eta=(\eta\wedge\dot{\eta}\wedge d\eta)\dot{\,}$.
This implies that the class is just a derivative of the Godbillon--Vey class and is not really a new one.
In this paper, we first introduce spaces of deformations of higher order and characteristic classes for such deformations.
These are generalization of \cite{12} and \cite{asuke:UT}.
In the study of characteristic classes, the Bott vanishing~\cite{Bott:LNM} and the Heitsch formula~\cite{Heitsch:Topology} are significant.
The former is the very basic in defining characteristic classes.
On the other hand, the latter describes derivatives of characteristic classes.
In particular, it largely divides the classes into those which vary according to deformations and rigid ones.
Characteristic classes introduced in this article well reflect such phenomena.
Finally, we will present a candidate of a characteristic class which is not a priori described by a combination of the Godbillon--Vey class and the Fuks--Lodder--Kotschick class.

This article is organized as follows.
First three sections are devoted to introduce spaces of deformations.
In the latter half, we introduce characteristic classes and study their properties.

\vskip\baselineskip\noindent
\textbf{I. Spaces of deformations}

\section{Definitions}
We fix a manifold $M$ and a foliation $\mathcal{F}$ of $M$.
Foliations are always assumed to be regular (without singularities) and of codimension $q$.

If $\{f_t\}$ is a family of some object smoothly parametrized by $t\in(-\epsilon,\epsilon)$, then a dot over an object denotes the derivative of the family at $t=0$.
Derivatives of order $k$ are also denoted by `${}^{(k)}$'.
For example, we set $f^{(1)}=\dot{f}=\left.\pdif{f}{t}\right|_{t=0}$.

\begin{definition}
Let $X,Y$ be topological spaces and $x\in X$, $y\in Y$.
A \textit{local mapping} from $(X,x)$ to $(Y,y)$ is a continuous mapping from a neighborhood of $x$ to $Y$ which maps $x$ to $y$.
If the mapping is a homeomorphism to the image, then we call the mapping a \textit{local homeomorphism}.
Similarly, we consider local mappings of class $C^r$ and local diffeomorphisms of class $C^r$.
\end{definition}

\begin{definition}
If $f\colon(X,x)\to(Y,y)$ is a local mapping of class $C^r$, then the $r$-jet of $f$ at $x$ is denoted by $j^r_xf$.
\end{definition}

\begin{definition}
Let $o\in\R^q$ be the origin.
The group of $r$-jets, denoted by $G^r$, is the set of the $r$-jets at $o$ of local diffeomorphism from $(\R^q,o)$ to $(\R^q,o)$ equipped with a product given by compositions.
\end{definition}

\begin{definition}
If $G$ is a Lie group, then $TG$ denotes its \textit{tangent group}.
\end{definition}

The tangent group is the tangent bundle of $G$ equipped with a natural group structure.
We have $TG=G\ltimes\mathfrak{g}$ as groups, where $\mathfrak{g}$ denotes the Lie algebra of $G$ and the multiplication is given by $(a,X)(b,Y)=(ab,\Ad_{b^{-1}}X+Y)$.
If $G\subset\GL_q(\R)$, then we have a matrix representation of $TG$.
Indeed, if we associate $(a,X)\in T\GL_q(\R)$ with $\begin{pmatrix}
a\\
aX & a
\end{pmatrix}$, then this gives a faithful representation.
A higher order analogue of this representation is given as follows.

\begin{definition}
For $r\in\N$, we define a subgroup $S^r$ of $\GL_{(r+1)q}(\R)$ by
\[
S^r=\left\{%
\begin{pmatrix}
a_0\\
a_1 & a_0\\
\vdots & \ddots & \ddots\\
a_r & \cdots & a_1 & a_0
\end{pmatrix}\,\middle|\;a_0\in\GL_q(\R),\ a_1,\ldots,a_r\in\gl_q(\R)\right\}.
\]
The Lie algebra of $S^r$ is denoted by $\mathfrak{s}^r$.
\end{definition}

The group $S^r$ appears as the structure group of a certain vector bundle $N^r$ defined in Section~\ref{sec2}.

\section{Normal bundles of higher order}
\label{sec2}
Let $T\CF\subset TM$ be the subbundle of $TM$ which consists of vectors tangent to a leaf.
We set $Q(\CF)=TM/T\CF$, and let $\nu\colon TM\to Q(\CF)$ and $p\colon Q(\CF)\to M$ be the projections.


\begin{definition}
Let $U\subset\R^q$ be an open set.
An $r$-frame at $x\in U$ is the $r$-jet $j^r_of$ a local diffeomorphism $f\colon(\R^q,o)\to(U,x)$ at $o$.
The bundle of $r$-frames at points in $U$ is denoted by $P^r(U)$.
\end{definition}

It is classical that $P^r(U)$ is a principal $G^r$-bundle over $U$.
Moreover, if $g\colon U\to V$ is a diffeomorphism of class $C^r$, then we have $g^*P^r(V)\cong P^r(U)$ as principal bundles.
Therefore, we can introduce the bundle of $r$-frames of $Q(\CF)$ as follows.
Let $\{U_\lambda\}$ be a foliation atlas, and $\pi_\lambda\colon U_\lambda\to T_\lambda$ be a submersion such that $\CF$ restricted to $U_\lambda$ is given by the fibers of $\pi_\lambda$.
Then, $\pi_\lambda^*P^r(T_\lambda)$ and $\pi_\mu^*P^r(T_\mu)$ are isomorphic on $U_\lambda\cap V_\lambda$.
Hence a principal $G^r$-bundle is formed on $M$.

\begin{definition}
Let $F^rQ(\CF)$ denote the principal $G^r$-bundle obtained as above.
We call $F^rQ(\CF)$ the \textit{bundle of $r$-frames} of $Q(\CF)$.
The natural projection from $F^rQ(\CF)$ to $F^{r-1}Q(\CF)$ is denoted by $\pi^r_{r-1}$.
\end{definition}

If $r=0$, then $F^0Q(\CF)$ is naturally diffeomorphic to $M$.
If the leaves of $\CF$ are points, then $F^rQ(\CF)=P^r(M)$.

Let $J^r_o(\R^q)$ be the space of $r$-jets of local mappings from $(\R,o)$ to $(\R^q,o)$, which is diffeomorphic to $(\R^q)^r$.
Then, $G^r$ naturally acts on $J^r_o(\R^q)$ by compositions on the left.

\begin{definition}
We set $J^r(\CF)=F^rQ(\CF)\times_{G^r}J^r_o(\R^q)$.
We call $J^r(\CF)$ \textit{the bundle of $r$-jets} transversal to $\CF$.
The natural projection from $J^r(\CF)$ to $J^{r-1}(\CF)$ is denoted by $p^r_{r-1}$.
\end{definition}

The bundles $J^0(\CF)$ and $J^1(\CF)$ are naturally isomorphic to $M$ and $Q(\CF)$, respectively.
Under these identifications, we have $p^1_0=p$.

\begin{remark}
We can regard elements of the bundle $J^r(\CF)$ as the jets of curves in the direction transversal to $\CF$.
Let $l\colon(\R,o)\to(M,x)$ be a curve.
Let $U\approx V\times T$ be a foliation chart which contains $x$.
By composing with the projection to $T$, we can regard $l$ as a curve, say $\overline{l}$, in $T$.
Then, the $r$-jet of $\overline{l}$ represents an element of $J^r(\CF)$.
Conversely, elements of $J^r(\CF)$ are represented in this way.
\end{remark}

\begin{definition}
We set $\pi^r=\pi^1_0\circ\pi^2_1\circ\cdots\circ\pi^r_{r-1}$ and $p^r=p^1_0\circ p^2_1\circ\cdots\circ p^r_{r-1}$.
\end{definition}

\begin{remark}
The fibers of the bundle $p^r\colon J^r(\CF)\to M$ are diffeomorphic to $(\R^q)^r$.
The bundle $J^r(\CF)$ is not a vector bundle if $r\geq2$, however, it admits the zero~section.
\end{remark}

We omit to write $\CF$ in what follows, for example, $Q$ will denote $Q(\CF)$.

We equip $J^r$ with a foliation as follows.
Let $U\approx V\times T$ be a foliation chart for $\CF$, where $V$ corresponds to leaves.
Then, $J^r$ is trivialized on $U$ as follows.
Let $(x,y)\in V\times T$ be local coordinates.
We consider coordinates $(x,y,\dot{y},\ldots,y^{(r)})$ on $V\times T\times(\R^q)^r$, where $y,\dot{y},\ldots,y^{(r)}$ are independent variables although each $y^{(k)}$ is meant to be the $k$-th derivative of $y$ with respect to a deformation parameter.
Let $\widehat{U}\approx\widehat{V}\times\widehat{T}$ also be a foliation chart and assume that $U\cap\widehat{U}\neq\varnothing$.
Suppose that the transition function is given by $(\widehat{x},\widehat{y})=(\varphi(x,y),\gamma(y))$.
Then we have $(\widehat{x},\widehat{y},\widehat{\dot{y}},\ldots,\widehat{y}^{(r)})=(\varphi(x,y),\gamma(y),D\gamma_y\dot{y},\ldots,,D^r\gamma(y,\dot{y},\ldots,y^{(r)}))$.
Therefore, we can define a foliation of $J^rQ$ by the condition that $y,\dot{y},\ldots,y^{(r)}$ are constants.
The foliation thus obtained is denoted by~$\CF^{(r)}$.
We have $\CF^{(0)}=\CF$.
Note that $F^rQ$ can be equipped with a foliation in a parallel way, which is denoted by~$\widehat\CF^{(r)}$.

\begin{definition}
We call $(x,y,\dot{y},\ldots,y^{(r)})$ the \textit{natural coordinates} for $J^r$.
\end{definition}

\begin{remark}
Even if $E\subset TM$ is not necessarily integrable, we can consider $F^1Q$, where $Q=TM/E$.
In order to define $J^rQ$ and $F^rQ$ for $r\geq2$, we need some auxiliary structures.
\end{remark}

\begin{definition}
Let $N^r=TJ^r/T\CF^{(r)}$ be the normal bundle of $\CF^{(r)}$.
The natural projection from $N^r$ to $N^{r-1}$ is denoted by $\widetilde{p}^r_{r-1}$.
We set $\widetilde{p}^r=\widetilde{p}^1_0\circ\widetilde{p}^2_1\circ\cdots\circ\widetilde{p}^r_{r-1}$.
The bundle $N^{r-1}$ is also denoted by $Q^{(r)}$.
\end{definition}

\begin{remark}
\begin{enumerate}
\item
We have $\CF^{(0)}=\CF$ and that $N^0=Q$.
\item
If $\CF$ is a foliation by points, then $Q^{(2)}=N^1$ coincides with the two-tangent bundle $T^{(2)}M=TTM$ of $M$.
\end{enumerate}
\end{remark}

We have a commutative diagram
\[
\begin{CD}
Q=N^0 @<{\widetilde{p}^1_0}<< N^1 @<<< \cdots @<<< N^{r-1} @<{\widetilde{p}^r_{r-1}}<< N^r \\
@V{p}VV @VVV @. @VVV @VVV \\
M\cong J^0 @<{p^1_0}<< Q\cong J^1 @<<< \cdots @<<< J^{r-1} @<{p^r_{r-1}}<< J^r\rlap{,}
\end{CD}
\]
where the vertical arrows are projections.

Let $U\approx V\times T$ be a foliation chart for $\CF$.
The, $J^r$ is trivialized as $U\times(\R^q)^r\approx V\times T\times(\R^q)^r$.
Together with the natural coordinates, this is a foliation chart for $\CF^{(r)}$.
Hence $N^r$ is trivialized as $U\times(\R^q)^r\times(\R^q)^{r+1}$.
Let $(x,y,\dot{y},\ldots,y^{(r)};v,\dot{v},\ldots,v^{(r)})$ be the naturally defined coordinates.
Note that $v,\dot{v},\ldots,v^{(r)}$ are independent variables, however, $v^{(k)}$ is meant to be the velocity which corresponds to $y^{(k)}$, where $y^{(0)}=y$.
The projection $\widetilde{p}^r\colon N^r\to Q$ is then locally given by
\[
\widetilde{p}^r(x,y,\dot{y},\ldots,y^{(r)};v,\dot{v},\ldots,v^{(r)})=(x,y;v).
\]

Let $U$ and $\widehat{U}$ be foliation charts.
We assume that $U\cap\widehat{U}\neq\varnothing$ and let $\gamma$ be the transition function in the transversal direction.
We define $T^kD\gamma_y$ to be the `time derivative' of $D\gamma_y$ of order $k$.
For example, we have $T^0D\gamma_y=D\gamma_y$, $T^1D\gamma_y=D^2\gamma_y\dot{y}$, $T^2D\gamma_y=D^3\gamma_y\dot{y}\dot{y}+D^3\gamma_y\ddot{y}$ and so on.
As $\widehat{v}=D\gamma_yv$, we have
\[
\widehat{v^{(k)}}=\sum_{i=0}^k\binom{k}{i}(T^iD\gamma_y)v^{(k-i)}.
\]
Similarly, we have $\widehat{\dot{y}}=D\gamma_y\dot{y}$ so that
\[
\widehat{y^{(k+1)}}=\sum_{i=0}^k\binom{k}{i}(T^iD\gamma_y)y^{(k-i+1)}.
\]
If we set $z^{(k)}=\frac1{k!}y^{(k)}$, $u^{(k)}=\frac1{k!}v^{(k)}$ and $\mathcal{D}^kD\gamma_y=\frac1{k!}T^kD\gamma_y$, then we~have
\begin{align*}
&\widehat{u^{(k)}}=\sum_{i=0}^k(\mathcal{D}^iD\gamma_y)u^{(k-i)},\\*
&(k+1)\widehat{z^{(k+1)}}=\sum_{i=0}^k(\mathcal{D}^iD\gamma_y)(k-i+1)z^{(k-i+1)}.
\end{align*}

Noticing that the frame $\left(\pdif{}{z^{(0)}},\ldots,\pdif{}{z^{(r)}}\right)$ corresponds to the fiber coordinates $(u^{(0)},\ldots,u^{(r)})$ for $N^r$, we introduce for later use the following

\begin{definition}[\cite{Bucataru}]
Let $\tau$ be an endomorphism of $N^r$ locally equal to $\sum_{k=1}^r\pdif{}{z^{(k)}}\otimes dz^{(k-1)}$ modulo $T\CF^{(r)}$.
We call $\tau$ an \textit{$r$-almost tangent structure} on $J^r$.
\end{definition}

In our contexts, the term `tangent' means that we regard $\tau$ as if it were an operator on the leaf space.

Let $Q'=\widetilde{p}^r{}^*N^0$ and $\nabla^{(r)}$ a Bott connection for $\mathcal{F}^{(r)}$.
If $v\in Q'$, then $v=\nu^r(\widetilde{v})$ for some $\widetilde{v}\in TJ^r$.
Since $\widetilde\nabla^{(r)}$ is a Bott connection, the horizontal lift of a vector in $T\CF$ belongs to $T\CF^{(r)}$.
Hence the horizontal lift of $v$ to $N^r$ is well-defined, which is denoted by $v^H$.

\begin{definition}
\label{def_a.3}
We call $v^H\in N^r$ the \textit{horizontal lift} of $v$.
\end{definition}

\begin{definition}
We set $H^{(r)}=\{v^H\in N^r\mid v\in Q'\}$ and call it the \textit{horizontal lift} of $Q'$.
\end{definition}

It is easy to see that $H^{(r)}$ is isomorphic to $Q'$ as vector bundles over $J^r$.
We have the following

\begin{theorem}
\label{thm_a.5}
If we set $H^{(k)}=\tau^{r-k}H^{(r)}$ for $0\leq k\leq r$, then we have $N^r\cong\bigoplus_{k=0}^rH^{(k)}$.
Moreover, each $H^{(k)}$ is isomorphic to $Q'=\widetilde{p}^r{}^*Q$.
\end{theorem}

The proof of Theorem~\ref{thm_a.5} is omitted because it is parallel to that of Theorem~3.1 in \cite{Bucataru} once we establish horizontal lifts as in Definition~\ref{def_a.3}.

\begin{remark}
If we begin with a Bott connection which is invariant under the holonomy, then we can apply the theory of prolongations~\cite{IY} to construct various lifts of objects on $Q=N^0$ to $F^rQ^r$ and to $N^r$.
In such a case, Theorem~\ref{thm_a.5} is naturally obtained by means of these lifts.
\end{remark}

\begin{remark}
Isomorphisms obtained by Theorem~\ref{thm_a.5} are homotopic because the structure group of $N^r$ over $J^r$ is homotopic to $\GL_q(\R)^r$.
In particular, if identifications are given by Bott connections for $\CF^{(r)}$, then a homotopy can be given by means of homotopies between connections.
\end{remark}

By Theorem~\ref{thm_a.5}, we can define a Bott connection for $\CF$ from a Bott connection $\nabla^{(r)}$ for $\CF^{(r)}$.
Let $p_k\colon N^r\to H^{(k)}$ be the projection given by the direct sum decomposition.
If $X\in TM$ and if $Y$ is a local section to $Q$, then we regard $Y$ as local section to $H^{(r)}$ and set $\underline{\nabla}^{(r)}_XY=p_r(\nabla^{(r)}_{X^H}Y)$.

\begin{remark}
We have $\underline{\nabla}^{(r)}_XY=p_{r-k}(\tau^k(\nabla^{(r)}_{X^H}Y))$ for $0\leq k\leq r$.
\end{remark}

\begin{definition}
Let $\nabla$ be a Bott connection for $\CF$.
A Bott connection $\nabla^{(r)}$ for $\CF^{(r)}$ is said to be a \textit{lift} of $\nabla$ if $\nabla=\underline{\nabla}^{(r)}$.
\end{definition}

Then, we have the following

\begin{proposition}
\label{prop_A.9}
Thus defined $\underline{\nabla}^{(r)}$ is indeed a Bott connection for $\CF$.
Moreover, $\underline{\nabla}^{(r)}$ is independent of $k$ and the lift of $Y$.
Conversely, if a Bott connection $\nabla$ for $\CF$ is given, then there is a lift of $\nabla$ to $\CF^{(r)}$.
\end{proposition}
\begin{proof}
The first claim is easy.
The second claim can be easily shown on foliation charts.
Then, a partition of unity argument shows the claim.
\end{proof}

\begin{remark}
Suppose that we work on holonomy invariant connections, which is the case if $\CF$ is Riemannian for example.
Then, we can find a unique lift of $\nabla$ which fulfills some naturalities by prolongations.
\end{remark}

\section{Spaces of deformations}
We introduce several spaces of deformations of foliations.

\begin{definition}
Let $\mathcal{F}$ be a foliation of a manifold $M$.
A one-parameter family $\{\mathcal{F}_t\}$ of foliations is \textit{of class $C^r$} if $\{T\mathcal{F}_t\}$ is of class $C^r$ as a family of distributions.
We always assume that $\mathcal{F}_0=\mathcal{F}$.
The space of $C^r$ one-parameter family is denoted by $\mathscr{D}^r(\mathcal{F})$.
We call a one-parameter family also an \textit{actual deformation} of $\CF$.
If $\{\mathcal{F}_t\}\in\mathscr{D}(\mathcal{F})$ and if $r\geq1$, then we regard $(r-1)$-jets of $\{T\mathcal{F}_t\}$ as $r$-jets of $\{\mathcal{F}_t\}$.
The space of $r$-jets is denoted by $D^r(\mathcal{F})$.
Noticing that there is a natural mapping from $D^r(\mathcal{F})$ to $D^{r-1}(\mathcal{F})$, we define the space of infinite jets of deformations by setting $D^\infty(\mathcal{F})=\varprojlim D^r(\mathcal{F})$.
\end{definition}

\begin{remark}
We can consider some reductions of spaces of deformations.
That is, we can define $\{\CF_t\}$ and $\{\CF'_t\}$ to be equivalent if there is a $C^r$ one-parameter family $\{\varphi_t\}$ of foliation preserving diffeomorphisms such that $\CF'_t=\varphi_t{}^*\CF_t$.
We can assume moreover that there is a smooth isotopies which connects $\varphi_t$ to the identity.
The characteristic classes we will later construct remain invariant under these equivalences.
On the other hand, the homotopy types of trivializations of normal bundles, if exist, are not invariant in general under former equivalences.
\end{remark}

If the normal bundle of $\CF=\CF_0$ is trivialized, then we fix trivializations of $\CF_t$ as follows.

\begin{definition}
\label{def3.2}
Suppose that the normal bundle $Q(\CF)$ of $\CF$ is trivialized by $e$.
Let $\{\CF_t\}$ be a $C^r$ one-parameter family of foliations with $\CF_0=\CF$.
Then, we fix a $C^r$ one-parameter family of identification $Q(\CF_t)\cong Q(\CF)$ and trivialize $Q(\CF_t)$ by the trivialization induced by $e$.
We say that such a family $\{\CF_t\}$ is \textit{framed}.
\end{definition}

There is a natural diagram
\stepcounter{theorem}
\[
\tag{\thetheorem}
\label{cd3.2}
\begin{gathered}
\xymatrix@C=15pt{%
\mathscr{D}^0(\CF) \ar[d] & \;\cdots\; \ar[l] & \mathscr{D}^{r-1}(\CF) \ar[l] \ar[d] & \mathscr{D}^r(\CF) \ar[l] \ar[d] & \;\cdots\; \ar[l] & \mathscr{D}^\infty(\CF) \ar[l] \ar[d] \\
D^0(\CF)=\{\CF\} & \;\cdots\; \ar[l] & D^{r-1}(\CF) \ar[l] & D^r(\CF) \ar[l] & \;\cdots\; \ar[l] & D^\infty(\CF). \ar[l]
}
\end{gathered}
\]

We first recall the canonical forms.
Let $F=F^1Q$ and $\pi_F\colon F\to M$ the projection.
An element of $F$ is the $1$-jet of a smooth mapping $u\colon(\R^q,o)\to(M,x)$ such that $\pi_U\circ u$ is a local diffeomorphism, where $U\approx V\times T$ is a foliation chart for $\CF$ which contains $x$ and  $\pi_U\colon U\to T$ is the projection.
Such an $u$ induces an isomorphism from $T_o\R^q$ to $Q_x$.
By using the standard identification of $T_o\R^q$ with $\R^q$, we consider $u$ as a mapping from $\R^q$ to $Q_x$.

\begin{definition}
If $u\in F$ and if $X\in T_uF$, then we set $\omega_u(X)=u^{-1}(\nu(\pi_F{}_*X))\in\R^q$.
We call $\omega$ the \textit{canonical form} on $P$.
\end{definition}

If $X\in\gl_q(\R)$, then a vector field on $F$ is defined by $X^*(u)=\left.\pdif{}{t}(u.\exp tX)\right|_{t=0}$.
We call $X^*$ the \textit{fundamental vector field} associated with $X$.
The canonical form is \textit{horizontal} in the sense that $\omega(X^*)=0$ for any $X\in\gl_q(\R)$.
The following is an~easy

\begin{lemma}
If $g\in\GL_q(\R)$, then $R_g^*\omega=g^{-1}\omega$, where $R_\bullet$ denotes the right action of\/ $\GL_q(\R)$ on $F$ and\/ $\GL_q(\R)$ acts on $\R^q$ standardly on the left.
\end{lemma}

Let $U$ be an open subset of $M$.
Then, local trivializations of $Q$ over $U$ and sections from $U$ to $F$ are in one-to-one correspondence.
Similarly, sections from $U$ to $F$ are in one-to-one correspondence to local trivializations of $Q^*$ on $U$.
Actually, if $\sigma$ is a section, then $\sigma^*\omega$ gives a local trivialization.
Conversely, let $\epsilon$ be a local trivialization and $e$ be its dual.
If we regard $e$ as a section from $U$ to $F$, then we have $e_i\epsilon^i(v)=\pi(v)$.
Hence we have $\epsilon=e^*\omega$.
Finally, let $\nabla$ be a Bott connection on $Q$ and $\theta$ the connection form on $F$, which is a $\gl_q(\R)$-valued one-form such that $R_g^*\theta=\Ad_{g^{-1}}\theta$ and that $\theta(X^*)=X$ for any $X\in\gl_q(\R)$.
If $\nabla$ is a Bott connection, then we have $d\omega+\theta\wedge\omega=0$ and vice versa.
We also call $\theta$ a Bott connection by abuse of notations.

Suppose now that $\{\CF_t\}$ is a smooth one-parameter family of foliations with $\CF=\CF_0$.
Then, $F_t=F^1Q(\CF_t)$ are isomorphic to $F$ so that we can find a smooth one-parameter family $\{\omega_t\}$ of canonical forms on $F$.
Similarly, we can find a smooth one-parameter family $\{\theta_t\}$ of Bott connections so that $d\omega_t+\theta_t\wedge\omega_t=0$.

If $\{\theta_t\}$ is a smooth one-parameter family of connection forms with $\theta_0=\theta$, then $\dot\theta$ is a $\gl_q(\R)$-valued one-form on $F$ such that $R_g\dot\theta=\Ad_{g^{-1}}\dot\theta$ and is horizontal.
Moreover, if $\{\theta_t\}$ are a family of Bott connections, then we have $d\dot{\omega}+\theta\wedge\dot{\omega}+\dot\theta\wedge\omega=0$.
If we set $\overline{\omega}^{(k)}=\frac1{k!}\left.\frac{\partial^k}{\partial t^k}\omega_t\right|_{t=0}$ and $\overline{\theta}^{(k)}=\frac1{k!}\left.\frac{\partial^k}{\partial t^k}\theta_t\right|_{t=0}$, then we have
\stepcounter{theorem}
\[
d\overline{\omega}^{(k)}+\sum_{i=0}^k\overline{\theta}^{(k-i)}\wedge\overline{\omega}^{(i)}=0.
\tag{\thetheorem}
\label{eq6.0}
\]
In a matrix form, we have
\[
d\begin{pmatrix}
\overline{\omega}^{(0)}\\
\overline{\omega}^{(1)}\\
\overline{\omega}^{(2)}\\
\vdots\\
\overline{\omega}^{(r)}
\end{pmatrix}+\begin{pmatrix}
\overline{\theta}^{(0)}\\
\overline{\theta}^{(1)} & \overline{\theta}^{(0)}\\
\overline{\theta}^{(2)} & \overline{\theta}^{(1)} & \overline{\theta}^{(0)}\\
\vdots & \ddots & \ddots & \ddots\\
\overline{\theta}^{(r)} & \overline{\theta}^{(r-1)} & \cdots & \overline{\theta}^{(1)} & \overline{\theta}^{(0)}
\end{pmatrix}\wedge\begin{pmatrix}
\overline{\omega}^{(0)}\\
\overline{\omega}^{(1)}\\
\overline{\omega}^{(2)}\\
\vdots\\
\overline{\omega}^{(r)}
\end{pmatrix}=0\rlap{.}
\]

We have the following

\begin{lemma}
Let $g\in\GL_q(\R)$.
\begin{enumerate}[label=\textup{\arabic*)}]
\item
For $k\geq0$, $\overline{\omega}^{(k)}$ is an $\R^q$-valued one-form on $F$ and is horizontal.
We have $R_g^*\overline{\omega}^{(k)}=g^{-1}\overline{\omega}^{(k)}$.
\item
For $k\geq1$, $\overline{\theta}^{(k)}$ is a $\gl_q(\R)$-valued one-form on $F$ and is horizontal.
We have $R_g\overline{\theta}^{(k)}=\Ad_{g^{-1}}\overline{\theta}^{(k)}$.
\end{enumerate}
\end{lemma}

On the other hand, the structure group of $N^r$ can be reduced to $\GL_q(\R)^r$ because it is contained in $S^r$.
Hence we have the following

\begin{proposition}
\label{prop6.6}
Let $FN^r$ be the frame bundle of $N^r$ and fix a reduction of the structure group to $\GL_q(\R)^r$.
Then, we can regard $\begin{pmatrix}
\overline{\omega}^{(0)} \\ \overline{\omega}^{(1)} \\ \vdots \\ \overline{\omega}^{r}\end{pmatrix}$ as a section to $FN^r$ and $\begin{pmatrix}
\overline{\theta}^{(0)}\\
\overline{\theta}^{(1)} & \overline{\theta}^{(0)}\\
\overline{\theta}^{(2)} & \overline{\theta}^{(1)} & \overline{\theta}^{(0)}\\
\vdots & \ddots & \ddots & \ddots\\
\overline{\theta}^{(r)} & \overline{\theta}^{(r-1)} & \cdots & \overline{\theta}^{(1)} & \overline{\theta}^{(0)}
\end{pmatrix}$ as a linear connection form on $FN^r$.
\end{proposition}

\begin{remark}
\begin{enumerate}
\item
A reduction in Proposition~\ref{prop6.6} is given by Theorem~\ref{thm_a.5}.
\item
Connections given by Proposition~\ref{prop6.6} depends on reductions, however, they are homotopic each other.
\item
The bundle $FN^r$ is different from $F^rQ$.
Indeed, $FN^r$ is a principal $S^r$-bundle while $F^rQ$ is a principal $G^r$-bundle.
\end{enumerate}
\end{remark}

Based on these observations, we introduce the following

\begin{definition}
\label{def3.9}
Let $\CF$ be a foliation of a manifold $M$ and $\omega$ the canonical form on $F=F^1Q$.
Given a Bott connection $\theta$ on $F$, we call a pair of families $\overline\omega=\{\overline\omega^{(k)}\}_{0\leq k\leq r}$ and $\overline\theta=\{\overline\theta^{(k)}\}_{0\leq k\leq r}$ \textit{infinitesimal deformation of $\omega$ and $\theta$} of order $r$ if the following conditions are satisfied:
\begin{enumerate}
\item
We have $\omega^{(0)}=\omega$.
If $k\geq1$, then $\overline{\omega}^{(k)}$ is an $\R^q$-valued one-form on $F$ and is horizontal.
In addition, $R_g^*\overline{\omega}^{(k)}=g^{-1}\overline{\omega}^{(k)}$ for $g\in\GL_q(\R)$.
\item
We have $\theta^{(0)}=\theta$.
If $k\geq1$, then $\overline{\theta}^{(k)}$ is a $\gl_q(\R)$-valued one-form on $F$ and is horizontal.
In addition, $R_g\overline{\theta}^{(k)}=\Ad_{g^{-1}}\overline{\theta}^{(k)}$ for $g\in\GL_q(\R)$.
\end{enumerate}
The space of infinitesimal deformations of $\theta$ of order $r$ is denoted by $\widehat{D}^r(\mathcal{F},\theta)$.
We set $\widehat{D}^r(\mathcal{F})=\bigcup_{\theta}\widehat{D}^r(\mathcal{F},\theta)$.
\end{definition}

\begin{remark}
We can always consider the trivial infinitesimal deformation, namely, an infinitesimal deformation such that $\overline\omega^{(k)}=0$ and $\overline\theta^{(k)}=0$ for any $k\geq1$.
\end{remark}

\begin{definition}
An infinitesimal deformation $\widehat\alpha$ of $\CF$ is said to be \textit{realizable} if there is a family of foliations $\{\CF_t\}$ of which the jet is equal to $\widehat\alpha$.
If the normal bundle of $\CF$ is trivialized by $e$, then a realization $\{\CF_t\}$ is said to be \textit{framed} if $\{\CF_t\}$ is framed as a deformation of $\CF$ of which the normal bundle is trivialized by $e$.
\end{definition}

\begin{remark}
When we consider framed deformations, only the homotopy types of trivializations are relevant.
In particular, characteristic classes which we will introduce later do not depend on the choice of families of identifications $Q(\CF_t)\cong Q(\CF)$ once the homotopy type of trivialization of $Q(\CF)$.
\end{remark}

By \eqref{eq6.0}, we have the following
\begin{lemma}
An infinitesimal deformation $(\overline{\omega},\overline{\theta})$ of order $r$ induces a Bott connection for $\CF^{(r)}$.
\end{lemma}

\begin{remark}
The space $\widehat{D}^1(\mathcal{F},\theta)$ is a vector space in the following sense.
Let $V$ be the space of infinitesimal deformations of order one of a Bott connection.
If $(\overline{\omega},\overline{\theta})$ and $(\overline{\omega}',\overline{\theta}')$ are infinitesimal deformations of order one, then we have
\begin{align*}
&d\overline{\omega}^{(1)}+\theta\wedge\overline{\omega}^{(1)}+\overline{\theta}^{(1)}\wedge\omega=0,\\
&d\overline{\omega}'{}^{(1)}+\theta\wedge\overline{\omega}'{}^{(1)}+\overline{\theta}'{}^{(1)}\wedge\omega=0\rlap{.}
\end{align*}
Hence we have $d(\overline{\omega}^{(1)}+\overline{\omega}'{}^{(1)})+\theta\wedge(\overline{\omega}^{(1)}+\overline{\omega}'{}^{(1)})+(\overline{\theta}^{(1)}+\overline{\theta}'{}^{(1)})\wedge\omega=0$.
This leads us to understand infinitesimal deformations of order one as elements of a certain cohomology~\cite{12},~\cite{13}.
Actually, if the order is equal to one, then characteristic classes are defined for cohomology classes determined by infinitesimal deformations~\cite{13},~\cite{asuke:UT}.
\end{remark}

\begin{remark}
If we only consider $\overline{\omega}$ in defining $\widehat{D}^r(\CF)$, we obtain a space of infinitesimal deformations of distributions.
This space is just sections to $(N^r)^*$ so that it enjoys linearity while $\widehat{D}^r(\CF)$ does not if $r\geq2$.
The difference is derived from the integrability of distributions which appear as the existence of Bott connections.
\end{remark}

\begin{remark}
There is a natural mapping from $D^r(\CF)$ to $\widehat{D}^r(\CF)$ for $0\leq r\leq\infty$ which extends the diagram~\eqref{cd3.2}.
\end{remark}

We introduce the following notions in order to facilitate arguments passing from families of foliations to infinitesimal deformations.

\begin{definition}
A polynomial on $M$ valued in $Q(\CF)$ of degree $r$ is said to be a \textit{formal section to $Q(\CF)$ of order $r$}.
A polynomial on $M$ valued in the space of Bott connections on $Q(\CF)$ of degree $r$ is said to be a \textit{formal Bott connection on $Q(\CF)$ of degree $r$}.
\end{definition}

The following is an easy

\begin{proposition}
Let $\overline\omega=\{\overline\omega^{(k)}\}_{0\leq k\leq r}$ and $\overline\theta=\{\overline\theta^{(k)}\}_{0\leq k\leq r}$ be an infinitesimal deformation of $\omega$ and $\theta$ of order $r$.
If we set $\widehat\omega=\sum_{k=0}^rt^k\overline\omega^{(k)}$ and $\widehat\theta=\sum_{k=0}^rt^k\overline\theta^{(k)}$, where $t^0=1$, then $\widehat\omega$ is a formal section to $Q(\CF)$ of order $r$ and $\widehat\theta$ is a formal Bott connection on $Q(\CF)$ of order $r$.
\end{proposition}

If $r=1$, then the space of infinitesimal deformations is formulated by Heitsch~\cite{12} as a cohomology which is denoted by $H^1(M;\Theta_\CF)$, .
The above constructions are compatible with it in the following sense.

\begin{theorem}
\label{thm3.17}
If $(\{\overline\omega^{(k)}\}_{k=0,1},\{\overline\theta^{(k)}\}_{k=0,1})\in\widehat{D}(\mathcal{F},\theta)$, then $\overline\omega^{(1)}$ is a infinitesimal deformation of $\omega_0$ with respect to $\theta=\theta^{(0)}$ (up to signatures) and $\theta^{(1)}$ is the infinitesimal deformation of $\theta$ with respect to $\omega_0$.
\end{theorem}
To be precise, we consider the restrictions of $\overline\omega^{(1)}$ to $\overline\theta^{(1)}$ to the tangent bundle of foliations in Theorem~\ref{thm3.17}.

\vskip\baselineskip\noindent
\textbf{II. Characteristic classes}

In what follows, deformations and realizations are assumed to be framed if the normal bundle of the foliation under consideration is trivialized.

\section{Basic definitions}
\label{sec4}

We discuss assuming that foliations are real and that normal bundles of them are trivialized.
We can argue in a parallel way even if we study foliations of which the normal bundles are non-trivial, or ones which admits transversal structures such as metric structure (Riemannian foliation) or holomorphic structure (transversely holomorphic foliation).
In order to study such foliations, we need to replace $\W_q$ by other ones such as $\WO_q$, $\W^\C$ and $\WU_q$, etc.
Note that if the normal bundle of the foliation under consideration is trivial, then we need to fix the homotopy type of the trivialization.

\begin{definition}
We set
\begin{align*}
&\gl_q(\R[t])_r=\{A_0+A_1t+\cdots+A_rt^r\mid A_i\in\gl_q(\R)\},\\*
&\GL_q(\R[t])_r=\{A_0+A_1t+\cdots+A_rt^r\mid A_0\in\GL_q(\R),\ A_i\in\gl_q(\R)\text{ if $i>0$}\},\\*
&\gl_q(\R[[t]])=\varprojlim_{r\in\N}\gl_q(\R[t])_r,\\*
&\GL_q(\R[[t]])=\varprojlim_{r\in\N}\GL_q(\R[t])_r.
\end{align*}
The multiplication in $\gl_q(\R[t])_r$ and the inverse in $\GL_q(\R[t])_r$ are taken modulo~$t^{r+1}$.
\end{definition}

The following is an easy

\begin{lemma}
For $X=\begin{pmatrix}
X_0\\
X_1 & X_0\\
\vdots & \ddots & \ddots\\
X_r & \cdots & X_1 &  X_0
\end{pmatrix}
\in\mathfrak{s}^r$, we set $\varphi(X)=X_0+X_1t+\cdots+X_rt^r$.
Then, $\varphi$ is an isomorphism from $\mathfrak{s}^r$ to $\gl_q(\R[t])_r$.
Moreover, the restriction of $\varphi$ to $S^r$ is isomorphism to $\GL_q(\R[t])_r$.
\end{lemma}
We set $X(t)=\varphi(X)$.

\begin{definition}
Let $X=\begin{pmatrix}
X_0\\
X_1 & X_0\\
\vdots & \ddots & \ddots\\
X_r & \cdots & X_1 &  X_0
\end{pmatrix}\in\mathfrak{s}^r$.
We represent $X^k$ as $X^k=\begin{pmatrix}
Y_0(k)\\
Y_1(k) & Y_0(k)\\
\vdots & \ddots & \ddots\\
Y_r(k) & \cdots & Y_1(k) &  Y_0(k)
\end{pmatrix}$ and set $C_{k,l}(X)=\frac1{k!}\left(\frac{-1}{2\pi}\right)^k\tr Y_l(k)$.
\end{definition}

Let $c_i$ be the $i$-th Chern polynomial, that is, we define $c_i$ to satisfy
\[
\det\left(\lambda I_q-\frac1{2\pi}X\right)=\lambda^q+\lambda^{q-1}c_1(X)+\cdots+\lambda c_{q-1}(X)+c_q(X),
\]
where $X\in\gl_q(\R)$.
We set $C_k(X)=\frac1{k!}\left(\frac{-1}{2\pi}\right)^k\tr X$, which corresponds to the $k$-th Chern character.
It is well-known that there uniquely exist polynomials $\varphi_k$ in free variables $x_1,\ldots,x_k$ such that
\[
c_k=\sum_{i=0}^k\varphi_k(C_1,2C_2\ldots,k!C_k)
\]
for $1\leq k\leq q$.
The polynomials $\varphi_k$ are often called the \textit{Newton polynomial}.
See~\cite{Husemoller} for details.

Let $x_i,\dot{x}_i,\ddot{x}_i,\ldots,x^{(m)}_i,\ldots$, where $1\leq i\leq q$, be free variables.
Let $\delta$ be the derivation on $\R[x^{(m)}_i]$ such that $\delta x^{(m)}_i=x^{(m+1)}_i$.
Then, each $\delta^l\varphi_k$ is a polynomial in $x^{(m)}_i$, where $1\leq i\leq k$ and $0\leq m\leq l$.
The polynomial $\delta^l\varphi_k$ evaluated by $x^{(m)}_i$ is denoted by $\delta^l\varphi_k(x_i,\dot{x}_i,\ldots,x^{(l)}_i)$.

\begin{definition}
\label{def4.4}
Let $X=\begin{pmatrix}
X_0\\
X_1 & X_0\\
\vdots & \ddots & \ddots\\
X_r & \cdots & X_1 &  X_0
\end{pmatrix}\in\mathfrak{s}^r$.
We set $C'_{k,l}(X)=k!C_{k,l}(X)$ and
\[
c_{k,l}(X)=\delta^l\varphi_k(C'_{i,0}(X),C'_{i,1}(X),\ldots,C'_{i,l}(X)).
\]
\end{definition}

We can describe the invariant polynomials on $\mathfrak{s}^r$ as follows.

\begin{theorem}
Let $I(S^r)$ be the algebra of invariant polynomials on $\mathfrak{s}^r$.
Then, we have $I(S^r)=\R[c_{k,l}]$, where $1\leq k\leq q$ and $0\leq l\leq r$.
\end{theorem}
\begin{proof}
For $X\in\gl_q(\R)$ and $1\leq i\leq q$, we set $\tau_i(X)=\tr X^q$.
If we extend $\tau_i$ naturally to $\gl_q(\R[t])_r$, then we have $I(\GL_q(\R[t])_r)\cong\R[\tau_1,\ldots,\tau_q]$.
If $X\in\mathfrak{s}^r$, then we have $\tau_k(X(t))=\sum_{l=0}^r\tr Y_l(k)t^l$.
Since $t$ is an independent, we are done.
\end{proof}

\section{Relevant differential graded algebras}
We introduce some differential graded algebras (DGA's for short).

\begin{definition}
We set
\[
\widetilde{\D^r\W}_q=\bigwedge[h_{i,(a)}]_{1\leq i\leq q,\ 0\leq a\leq r}\otimes\R[c_{i,(b)}]_{1\leq i\leq q,\ 0\leq b\leq r}.
\]
We allow $r$ to be $\infty$.
\begin{enumerate}[i)]
\item
We define an exterior derivative $d$ on $\widetilde{\D^r\W}_q$ by the conditions $dh_{i,(j)}=c_{i,(j)}$ and $dc_{i,(j)}=0$.
We formally set $h_{i,(-1)}=c_{i,(-1)}=0$ for all $i$, and define a (unsigned) derivation $\widetilde\delta$ on $\widetilde{\D^r\W}_q$ by the conditions that
\begin{align*}
&\widetilde{\delta}(c_{i,(a)})=c_{i,(a+1)},\\*
&\widetilde{\delta}(h_{i,(b)})=h_{i,(b+1)},
\end{align*}
where $a<r$.
\item
We set $\ord(c_{i,(a)})=a$, $\ord(h_{i,(b)})=b$ and naturally extend $\ord$ to the whole $\widetilde{\D^r\W}_q$.
\item
We set $\deg(c_{i,(a)})=2$, $\deg(h_{i,(b)})=1$ and naturally extend $\deg$ to the whole $\widetilde{\D^r\W}_q$.
\end{enumerate}
\end{definition}
We set $h_i=h_{i,(0)}$ and $c_i=c_{i,(0)}$.
Note that $d$ and $\widetilde{\delta}$ commute each other.
In what follows, we do not explicitly indicate exterior products, namely, we omit the symbol~`$\wedge$'.

\begin{definition}
We set
\begin{align*}
&\widetilde\sigma(c_{i,(l)})=\frac12l(l-1)c_{i,(l-1)},\\*
&\widetilde\sigma(h_{i,(l)})=\frac12l(l-1)h_{i,(l-1)}.
\end{align*}
We extend $\widetilde\sigma$ to the whole $\widetilde{\D^r\W}_q$ as a (unsigned) derivation.
\end{definition}

\begin{lemma}
\label{lem_new6.2}
If $c\in\widetilde{\D^r\W_q}$ is a monomial and if $\ord(c)<r$, then we have $\widetilde\sigma\circ\widetilde\delta(c)-\widetilde\delta\circ\widetilde\sigma(c)=\ord(c)c$.
\end{lemma}
\begin{proof}
The claim holds if $c=c_{i,(l)}$ or if $c=h_{i,(l)}$.
Indeed, we have
\begin{align*}
\widetilde\sigma\circ\widetilde\delta(c_{i,(l)})-\widetilde\delta\circ\widetilde\sigma(c_{i,(l)})&=\widetilde\sigma(c_{i,(l+1)})-\frac12l(l-1)\widetilde\delta(c_{i,(l-1)})\\*
&=\frac12(l+1)lc_{i,(l)}-\frac12l(l-1)c_{i,(l)}\\*
&=lc_{i,(l)}.
\end{align*}
Similar calculations are valid for $h_{i,(l)}$.
Suppose that the claim holds for monomials of degree less than or equal to $k$.
Let $c\in\widetilde{\D^r\W_q}$ be a monomial of degree $k+1$.
If $c$ admits a decomposition that $c=c_1c_2$ with $\ord(c_1)\leq k$ and $\ord(c_2)\leq k$, then we have
\begin{align*}
\stepcounter{theorem}
\tag{\thetheorem}
\label{eq_new6.3}
&\hphantom{{}={}}%
\widetilde\sigma\circ\widetilde\delta(c)-\widetilde\delta\circ\widetilde\sigma(c)\\*
&=\widetilde\sigma(\widetilde\delta(c_1)c_2+c_1\widetilde\delta(c_2))-\widetilde\delta(\widetilde\sigma(c_1)c_2+c_1\widetilde\sigma(c_2))\\*
&=\widetilde\sigma(\widetilde\delta(c_1))c_2+\widetilde\delta(c_1)\widetilde\sigma(c_2)+\widetilde\sigma(c_1)\widetilde\delta(c_2)+c_1\widetilde\sigma(\widetilde\delta(c_2))\\*
&\hphantom{{}={}}%
-\widetilde\delta(\widetilde\sigma(c_1))c_2-\widetilde\sigma(c_1)\widetilde\delta(c_2)-\widetilde\delta(c_1)\widetilde\sigma(c_2)-c_1\widetilde\delta(\widetilde\sigma(c_2))\\*
&=\ord(c_1)c_1c_2+\ord(c_2)c_1c_2\\*
&=\ord(c)c.
\end{align*}
Otherwise, we have $c=c_1c_2$ with $c_1=c_{i,(k+1)}$ or $c_1=h_{i,(k+1)}$, and that $\ord(c_2)=0$.
Then, by the similar arguments as above, we can show that the claim also holds for $c$.
\end{proof}

\begin{remark}
We set $\widehat{h}_i=\sum_{l=0}^{+\infty}\frac1{l!}t^lh_{i,(l)}$.
If we termwise apply $\widetilde\delta$, we have $\widetilde\delta\widehat{h}_i=\sum_{l=0}^{+\infty}\frac1{l!}t^lh_{i,(l+1)}$.
On the other hand, we have $\pdif{}{t}\widehat{h}_i=\sum_{l=0}^{+\infty}\frac1{l!}t^lh_{i,(l+1)}$.
In this sense, we have $\widetilde\delta=\pdif{}{t}$.
Similarly, we have $\widetilde\sigma=\frac12t^2\pdif{}{t}$.
As this correspondence is contravariant, $\widetilde\sigma\circ\widetilde\delta-\widetilde\delta\circ\widetilde\sigma$ corresponds to $\pdif{}{t}\circ\left(\frac12t^2\pdif{}{t}\right)-\left(\frac12t^2\pdif{}{t}\right)\circ\pdif{}{t}$.
Indeed, let $f$ be a function in $t$.
Then, we have $\pdif{}{t}\circ\left(\frac12t^2\pdif{}{t}\right)(f)-\left(\frac12t^2\pdif{}{t}\right)\circ\pdif{}{t}(f)=t\pdif{f}{t}$.
This operator just corresponds to $h_{i,(l)}\mapsto lh_{i,(l)}$.
\end{remark}

Let $\widetilde{\D^r\W_q}{}^l$ be the subspace of $\widetilde{\D^r\W_q}$ which is generated by monomials of order $l$.
If $l<0$, then we set $\widetilde{\D^r\W_q}{}^l=\{0\}$.
We have the following

\begin{lemma}
\begin{enumerate}
\item
If $l<r$, then we have $\widetilde\delta(\widetilde{\D^r\W_q}{}^l)\subset\widetilde{\D^r\W_q}{}^{l+1}$.
\item
We have $\widetilde\sigma(\widetilde{\D^r\W_q}{}^l)\subset\widetilde{\D^r\W_q}{}^{l-1}$.
\end{enumerate}
\end{lemma}

We will make use of the following variants of $\widetilde\sigma$.

\begin{definition}
We set
\begin{align*}
&\widetilde\sigma'(c_{i,(l)})=lc_{i,(l-1)},\\*
&\widetilde\sigma'(h_{i,(l)})=lh_{i,(l-1)}.
\end{align*}
We extend $\widetilde\sigma'$ to the whole $\widetilde{\D^r\W_q}$ as a (unsigned) derivation.
\end{definition}

\begin{definition}
\label{def5.8}
Let $c=h_{i_1,(l_1)}\cdots h_{i_a,(l_a)}c_{j_1,(k_1)}\cdots c_{j_b,(k_b)}$, where $(i_a,(l_a))\neq(i_{a'},(l_{a'}))$ if $a\neq a'$.
We set $l(c)=a+b$ and call $l(c)$ the \textit{length} of $c$.
The subspace of $\widetilde{\D^r\W_q}$ generated by elements of length $l$ is denoted by $\widetilde{L}_l$.
\end{definition}

We have the following lemma which corresponds to Lemma~\ref{lem_new6.2}.

\begin{lemma}
\label{lem6.9}
If $c\in\widetilde{\D^r\W_q}$ is a monomial and if $l(c)<r$, then we have $\widetilde\delta\circ\widetilde\sigma'(c)-\widetilde\sigma'\circ\widetilde\delta(c)=l(c)c$.
\end{lemma}

\section{Bott vanishing theorem and Heitsch formula}
The most basic characteristic classes for foliations are Pontrjagin classes of the normal bundles.
Usually we only consider the Pontrjagin classes of degree $4k$, however, in order to introduce secondary classes, we also consider the classes of degree $4k+2$.
Let $c_k$ denote the $k$-th Chern polynomial, which will corresponds to the Pontrjagin form of degree $2k$.

\begin{definition}
\label{def6.1}
If $c=c_1^{j_1}\cdots c_q^{j_q}\in\R[c_1,\ldots,c_q]$, then we set $\norm{c}=j_1+2j_1+\cdots+qj_q$.
\end{definition}
We will extend $\norm{\;\cdot\;}$ later in Definition~\ref{def7.10}.

\begin{definition}
Let $I_\CF=\{\omega\in\Omega^*(M)\mid\text{$\omega$ is locally of the form $\textstyle{\sum}\,\omega_i\wedge dy^i$}\}$, where $(y^1,\ldots,y^q)$ are local coordinates in the transversal direction.
Elements of $I_\CF^k$ are said to be of \textit{transveresal degree greater than or equal to $k$}, or of transversal degree greater than $k-1$.
\end{definition}

The following is a fundamental result by Bott.

\begin{theorem}[Bott~\cite{Bott:LNM}]
Let $\nabla$ be a Bott connection for $\CF$.
For $c\in\R[c_1,\ldots,c_q]$, $c(\nabla)$ denotes the Pontrjagin form which is calculated by we using $\nabla$.
If $\norm{c}=k$, then $c(\nabla)\in I_\CF^k$.
In particular, $c(\nabla)$ is trivial as a differential form if $\norm{c}>q$.
\end{theorem}

\begin{remark}
\begin{enumerate}
\item
The differential form $c_i(\nabla)$ is always exact if $i$ is odd.
If the normal bundle is trivial, then $c_i(\nabla)$ are also exact for $i$ even.
\item
If foliations under considerations are transversely holomorphic, then we make use of the Chern classes and $c_i$ denotes the $i$-th Chern class.
In this case, only the imaginary parts of $c_i(\nabla)$ are exact so that we have to modify arguments.
\end{enumerate}
\end{remark}

Next, we recall the Heitsch formula for deformations~\cite{Heitsch:Topology}.
Let $\{\CF_t\}$ be a one-parameter $C^1$-family of foliations with $\CF_0=\CF$.
We fix identifications of $Q(\CF_t)$ with $Q(\CF)$, and let $\{\theta_t\}$ be a family of Bott connections with $\theta_0=\theta$.
If $Q(\CF)$ is trivial, then we fix a trivialization, say $e$, and let $\theta^f$ be the flat connection with respect to $e$.
If we do not assume the triviality, then let $\theta^m$ a metric connection with respect to a Riemannian metric on $Q(\CF)$.
We represent by $\theta^\natural$ either $\theta^f$ or $\theta^m$ according to the assumption.
Usually we assume that $Q(\CF)$ is trivialized so that $\theta^\natural$ denotes $\theta^f$.

\begin{definition}
If $G$ is a Lie group, then $I^k(G)$ denotes the algebra of $\Ad_G$-invariant symmetric multilinear mappings on $\mathfrak{g}$ of degree $k$.
If $c$ is an $\Ad_G$-invariant polynomial of degree $k$ on $\mathfrak{g}$, then the polarization of $c$ is also denoted by $c$ by abuse of notations.
\end{definition}

\begin{remark}
If $c\in I^k(G)$ and if $f_c$ denotes the polarization of $c$, then we have
\[
f_c(X_1,\ldots,X_k)=\left.\frac1{k!}\pdif{}{t_1}\cdots\pdif{}{t_k}c(t_1X_1+\cdots+t_kX_k)\right|_{t_1=\cdots=t_k=0}
\]
for $X_1,\ldots,X_k\in\mathfrak{g}$, where $t_1,\ldots,t_k$ are free variables.
We have $f_c(X,\ldots,X)=c(X)$.
\end{remark}

\begin{definition}
Let $f\in I^k(\GL_q(\R))$.
We set $\theta_{t,s}=s\theta_t+(1-s)\theta^\natural$, $\Omega_{t,s}=d\theta_{t,s}+\theta_{t,s}\wedge\theta_{t,s}$ and
\begin{align*}
&\psi_t=\pdif{}{t}\theta_t,\\*
&\Delta_f(\theta_t,\theta^\natural)=k\int_0^1f(\theta_t-\theta^\natural,\Omega_{t,s},\ldots,\Omega_{t,s})ds,\\*
&V_f(\theta_t,\theta^\natural)=\int_0^1sf(\psi_t,\theta_t-\theta^\natural,\Omega_{t,s},\ldots,\Omega_{t,s})ds.
\end{align*}
\end{definition}

Note that $\theta_{t,0}=\theta^\natural$ and $\theta_{t,1}=\theta_t$.
Hence we have $\Omega_{t,0}=d\theta^\natural+\theta^\natural\wedge\theta^\natural$ and $\Omega_{t,1}=d\theta_t+\theta_t\wedge\theta_t$.

\begin{remark}
The differential form $\left(\frac{-1}{2\pi}\right)^i\Delta_{c_i}(\theta_t,\theta^\natural)$ will be the image of $h_i$ by the characteristic mapping.
See Definition~\ref{def7.2}
\end{remark}

The following is shown by Heitsch.

\begin{theorem}[\cite{Heitsch:Topology}]
\label{thm6.9}
The following formulae hold\textup{:}
\begin{align*}
&\pdif{}{t}\Delta_f(\theta_t,\theta^\natural)=k(k-1)dV_f(\theta_t,\theta^\natural)+kf(\psi_t,\Omega_{t,1},\ldots,\Omega_{t,1}),\\*
&\pdif{}{t}d(\Delta_f(\theta_t,\theta^\natural))=\pdif{}{t}f(\Omega_{t,1},\ldots,\Omega_{t,1})=kdf(\psi_t,\Omega_{t,1},\ldots,\Omega_{t,1}).
\end{align*}
\end{theorem}

\begin{definition}
If $f\in I^k(\GL_q(\R))$ and if $r\geq1$, then we set
\begin{align*}
&\alpha^{(r)}_f(\theta_t,\theta^\natural)=\dfrac{\partial^{r-1}}{\partial t^{r-1}}f(\psi_t,\Omega_{t,1},\ldots,\Omega_{t,1}),\\*
&V^{(r)}_f(\theta_t,\theta^\natural)=\dfrac{\partial^{r-1}}{\partial t^{r-1}}V_f(\theta_t,\theta^\natural).
\end{align*}
We set $\alpha^{(0)}_f(\theta_t,\theta^\natural)=\Delta_{c_i}(\theta_t,\theta^\natural)$.
\end{definition}

The following is immediate.

\begin{lemma}
\label{lem6.11}
We have $\alpha^{(r)}_f(\theta_t,\theta^\natural)\in I_\CF^{k-r}$.
\end{lemma}

\begin{corollary}
If $r\geq1$, then we have
\[
\dfrac{\partial^r}{\partial t^r}\Delta_f(\theta_t,\theta^\natural)=k(k-1)dV^{(r)}_f(\theta_t,\theta^\natural)+k\alpha_f^{(r)}(\theta_t,\theta^\natural).
\]
\end{corollary}

An infinitesimal version of above formulae is given as follows.

\begin{corollary}
\label{cor6.5}
The following formulae hold\textup{:}
\begin{align*}
&\left.\dfrac{\partial^r}{\partial t^r}\Delta_f(\theta_t,\theta^\natural)\right|_{t=0}=k(k-1)dV_f^{(r)}(\theta,\theta^\natural)+k\alpha_f^{(r)}(\dot{\theta},\Omega,\ldots,\Omega),\\*
&\left.\dfrac{\partial^r}{\partial t^r}d(\Delta_f(\theta_t,\theta^\natural))\right|_{t=0}=\left.\dfrac{\partial^r}{\partial t^r}f(\Omega_{t,1},\ldots,\Omega_{t,1})\right|_{t=0}=kd\alpha_f^{(r)}(\dot\theta,\Omega,\ldots,\Omega),
\end{align*}
where $\dot\theta=\left.\pdif{}{t}\theta_t\right|_{t=0}$, $\Omega=d\theta+\theta\wedge\theta$ and $r\geq1$.
\end{corollary}

\section{Characteristic mapping}
Characteristic classes for deformations of foliations are realized as follows.
If $\widehat\alpha\in\widehat{D}^r(\CF)$, then $\widehat\alpha$ induces a Bott connection, say $\nabla^{(r)}$, on $N^r$.
If $Q(\CF)$ is trivialized, then we consider the induced trivialization of $N^r$ and let $\nabla_f^{(r)}$ be the flat connection.
Otherwise, we fix a Riemannian metric on $N^r$ and let $\nabla_m^{(r)}$ be a metric connection.
We set $\nabla_\natural^{(r)}$ to be $\nabla_f^{(r)}$ or $\nabla_m^{(r)}$.

We briefly recall the Chern--Simons forms.
Let $f$ be a polynomial in $c_{i,(l)}$, and let $\theta$ and $\theta'$ be connection forms.
We set $\theta_t=(1-t)\theta+t\theta'$ and let $R_t$ be the curvature form of $\theta_t$.
Then, $f(R_t)$ is closed and we have $f(R_t)=a+b\wedge dt$ for some $a$ and $b$ which do not involve $dt$.

\begin{definition}
We set $\Delta_f(\theta,\theta')=\int_0^1b\,dt$.
\end{definition}
It is classical that $d\Delta_f(\theta,\theta')=f(R(\theta),\ldots,R(\theta))-f(R(\theta'),\ldots,R(\theta'))$, where $R(\theta)$ and $R(\theta')$ denote the curvature matrix of $\theta$ and $\theta'$, respectively.

\begin{definition}
\label{def7.2}
Let $R$ be the curvature matrix of $\nabla^{(r)}$ and we set
\begin{align*}
&\widetilde\chi'_{\widehat\alpha}(c_{i,(l)})=\left(\frac{-1}{2\pi}\right)^ic_{i,(l)}(R),\\*
&\widetilde\chi'_{\widehat\alpha}(h_{i,(l)})=\left(\frac{-1}{2\pi}\right)^i\Delta_{c_{i,(l)}}(\nabla^{(r)},\nabla_\natural^{(r)}).
\end{align*}
It is well-known that we can regard these are differential forms on the base space $\Omega^*(J^{(r)})$.
We extend $\widetilde\chi'_{\widehat\alpha}$ to $\widetilde{\D^\infty\W}_q$ as a homomorphism to $\Omega^*(J^{(r)})$.
Pulling back by the trivial section from $M$ to $J^{(r)}$, we obtain a homomorphism to $\Omega^*(M)$, which is also denoted by $\widetilde\chi'_{\widehat\alpha}$ by abuse of notations.
\end{definition}

\begin{definition}
We set
\begin{align*}
&\widetilde\chi_{\widehat\alpha}(c_{i,(l)})=\left(\frac{-1}{2\pi}\right)^ic_{i,(l)}(R),\\*
&\widetilde\chi_{\widehat\alpha}(h_{i,(l)})=\left(\frac{-1}{2\pi}\right)^ii\alpha^{(l)}_{c_i}(\nabla^{(r)},\nabla_\natural^{(r)}).
\end{align*}
\end{definition}

Noticing that the correspondences in Theorem~\ref{thm6.9} and Corollary~\ref{cor6.5} are functorial with respect to mapping transversal to foliations and pull-backs, we extend $\norm{\;\cdot\;}$ in Definition~\ref{def6.1} as follows.

\begin{definition}
\label{def7.10}
We set $\norm{c_{i,(l)}}=\max\{i-l,0\}$.
We set $\norm{h_{i,(0)}}=0$ and $\norm{h_{i,(l)}}=\max\{i-l,0\}$ if $l\geq1$.
\end{definition}
We extend $\norm{\;\cdot\;}$ to $\widetilde{\D^r\W_q}$ as in defining the degree of polynomials.

\begin{definition}
\label{relations}
\begin{enumerate}
\item
We set $\mathscr{I}_0=\{c\in\widetilde{\D^r\W_q}\mid\norm{c}>q\}$.
The ideal of $\widetilde{\D^r\W_q}$ generated by $\bigcup_{k\in\N}\widetilde\delta^k\mathscr{I}_0$ is denoted by $\mathscr{I}$.
\item
We set $\mathscr{I}'_0=\{c_1^{j_1}\cdots c_q^{j_q}\mid j_1+2j_2+\cdots+qj_q>q\}\subset\mathscr{I}_0$.
The ideal generated by $\bigcup_{k\in\N}\widetilde\delta^k\mathscr{I}'_0$ is denoted by $\mathscr{I}'$.
\end{enumerate}
\end{definition}

\begin{remark}
The ideal $\mathscr{I}'$ corresponds to elements which vanish by the Bott vanishing and their derivatives.
Taking also the Heitsch formula into account, we obtain the ideal $\mathscr{I}$.
We have $\mathscr{I}'=\mathscr{I}$ if and only if $q=1$.
\end{remark}

Since $\widetilde\sigma$ does not decrease the values of $\norm{\;\cdot\;}$, we have the following

\begin{lemma}
The derivation $\widetilde\sigma$ leaves $\mathscr{I}$ and $\mathscr{I}'$ invariant.
\end{lemma}

Similarly, we have the following

\begin{lemma}
\label{lem6.9__}
The derivation $\widetilde\sigma'$ leaves $\mathscr{I}'$ invariant.
\end{lemma}

Note that $\widetilde\sigma'$ does \textit{not} preserve $\mathscr{I}$.

The homomorphisms $\widetilde\chi$ and $\widetilde\chi'$ induce homomorphisms from $\widetilde{\D^r\W_q}$ to $H_{\mathrm{DR}}^*(M;\R)$, which is also denoted again by $\widetilde\chi$ and $\widetilde\chi'$, respectively.

\begin{lemma}
\label{lem7.14}
\begin{enumerate}[label=\textup{\arabic*)}]
\item
The mapping $\widetilde\chi'\colon\widetilde{\D^r\W_q}\to H_{\mathrm{DR}}^*(M;\R)$ vanishes on $\mathscr{I}'$.
\item
The mapping $\widetilde\chi\colon\widetilde{\D^r\W_q}\to H_{\mathrm{DR}}^*(M;\R)$ vanishes on $\mathscr{I}$.
\end{enumerate}
\end{lemma}
\begin{proof}
By the construction, $\widetilde\chi$ and $\widetilde\chi'$ vanish on $\mathscr{I}'$.
By Theorem~\ref{thm6.9} and Lemma~\ref{lem6.11}, $\widetilde\chi$ vanishes on $\mathscr{I}$.
\end{proof}

\begin{definition}
We set
$\D^r\W_q=\widetilde{\D^r\W_q}/\mathscr{I}$ and $\D^r\W_q'=\widetilde{\D^r\W_q}/\mathscr{I}'$.
The homomorphisms induced by $\widetilde\chi$ and $\widetilde\chi'$ from $\D^r\W_q'\to H_{\mathrm{DR}}^*(M;\R)$ are denoted by $\chi^b$ and $\chi'$, respectively.
The homomorphism induced by $\widetilde\chi$ from $\D^r\W_q\to H_{\mathrm{DR}}^*(M;\R)$ is denoted by $\chi$.
The natural homomorphism from $H^*(\D^r\W_q')$ to $H^*(\D^r\W_q)$ is denoted by $\tau$.
\end{definition}

Typical elements of $H^*(\D^r\W_q)$ are given as follows.

\begin{definition}[\cite{Godbillon-Vey}, \cite{Fuks}, \cite{Lodder}, \cite{Kotschick}]
The class $h_{1,(0)}c_{1,(0)}{}^q\in H^{2q+1}(\D^r\W_q)$ is called the \textit{Godbillon--Vey class} and denoted by $\GV$.
The class $h_{1,(0)}h_{1,(1)}c_{1,(0)}{}^q\in H^{2q+2}(\D^r\W_q)$ is called the \textit{Fuks--Lodder--Kotschick class} (the FLK class for short) and denoted by $\FLK$.
If we clarify $q$, then $\GV$ and $\FLK$ are denoted by $\GV_q$ and $\FLK_q$, respectively.
\end{definition}

Note that $\D^r\W_1=\D^r\W_1'$ while $\D^r\W_q\neq\D^r\W_q'$ for $q\geq2$.
As $\widetilde\delta$ and $\widetilde\sigma$ leave $\mathscr{I}$ invariant so that derivations are induced on $H^*(\D^r\W_q)$, which are denoted by $\delta$ and $\sigma$, respectively.
Similarly,  $\widetilde\delta$ and $\widetilde\sigma'$ induce derivations on $H^*(\D^r\W_q')$, which are denoted by $\delta'$ and $\sigma'$.

\begin{lemma}
\label{lem7.4}
The homomorphism $\chi^b$ is cochain homotopic to $\chi'$.
\end{lemma}
\begin{proof}
We introduce an order in $\{c_{i,(l)},h_{j,(k)}\}_{1\leq i,j\leq q,\ 0\leq l,k\leq r}$ by requiring that
\begin{enumerate}
\item
$c_{i,(l)}<h_{j,(k)}$ for any $i,j,l,k$.
\item
$c_{i,(l)}<c_{j,(k)}$ if $i<j$ or if $i=j$ and $l<k$.
\item
$h_{i,(l)}<h_{j,(k)}$ if $i<j$ or if $i=j$ and $l<k$.
\end{enumerate}
Let $\varphi\in\widetilde{\D^r\W^q}$ and assume that $\varphi=\varphi_1\cdots\varphi_m$, where $\varphi_1,\ldots,\varphi_a\in\{c_{i,(l)},h_{j,(k)}\}$ and $\varphi_1\leq\varphi_2\leq\cdots\leq\varphi_a$.
Note that such a representation of $\varphi$ is unique if possible.
We~set
\begin{align*}
&H_0(c_{i,(l)})=0,\\*
&H_0(h_{i,(l)})=\left(\frac{-1}{2\pi}\right)^ii(i-1)V^{(l)}_{c_i}(\theta_t,\theta^\natural),
\end{align*}
and
\[
H(\varphi)=\sum_{i=1}^a(-1)^{\deg\varphi_1\cdots\varphi_{i-1}}\widetilde\chi(\varphi_1\cdots\varphi_{i-1})H_0(\varphi_i)\widetilde\chi'(\varphi_{i+1}\cdots\varphi_m),
\]
where $\deg c_{i,(l)}=2$ and $\deg h_{i,(l)}=1$.
We extend $H$ to the whole $\widetilde{\D^r\W_q}$ by linearity.

We have
\begin{align*}
&(H_0d+dH_0)(c_{i,(l)})=0=\widetilde{\chi}(c_{i,(l)})-\widetilde{\chi}'(c_{i,(l)}),\\*
&(H_0d+dH_0)(h_{i,(l)})=\left(\dfrac{-1}{2\pi}\right)^ii(i-1)dV^{(l)}_{c_i}(\theta_t,\theta^\natural)%
=\widetilde\chi(h_{i,(l)})-\widetilde\chi'(h_{i,(l)}).
\end{align*}
Suppose that $\varphi=\varphi_1\psi$, where $\varphi_1\in\{c_{i,(l)},h_{j,(k)}\}$, and that $(Hd+dH)(\psi)=\widetilde\chi(\psi)-\widetilde\chi'(\psi)$.
Then, we have
\begin{align*}
&\hphantom{{}={}}%
H(d\varphi)+d(H\varphi)\\*
&=H((d\varphi_1)\psi+(-1)^d\varphi_1d\psi)+d((H\varphi_1)\widetilde\chi'(\psi)+(-1)^d\widetilde\chi(\varphi_1)H(\psi))\\*
&=\begin{aligned}[t]%
&H(d\varphi_1)\widetilde\chi'(\psi)-(-1)^d\widetilde\chi(d\varphi_1)H(\psi)+(-1)^dH(\varphi_1)\widetilde\chi'(d\psi)+\widetilde\chi(\varphi_1)H(d\psi)\\*
&+d(H\varphi_1)\widetilde\chi'(\psi)-(-1)^d(H\varphi_1)d\widetilde\chi'(\psi)+(-1)^d(d\widetilde\chi\varphi_1)H\psi+\widetilde\chi(\varphi_1)dH(\psi)
\end{aligned}\\*
&=(Hd+dH)(\varphi_1)\widetilde\chi'(\psi)+\widetilde\chi(\varphi_1)(Hd+dH)(\psi)\\*
&=\widetilde\chi(\varphi_1)\widetilde\chi'(\psi)-\widetilde\chi'(\varphi)+\widetilde\chi(\varphi)-\widetilde\chi(\varphi_1)\widetilde\chi'(\psi)\\*
&=\widetilde\chi(\varphi)-\widetilde\chi'(\varphi),
\end{align*}
where $d=\deg\varphi_1$.
Thus $H$ is a cochain homotopy.
As $H$ preserves $\mathscr{I}'$, a cochain homotopy from $\chi'$ to $\chi^b$ is induced by $H$.
\end{proof}

By the construction, we have the following
\begin{theorem}
We have $\chi\circ\tau=\chi'$.
\end{theorem}

\begin{proposition}
The natural inclusion of\/ $\W_q$ into $\D^r\W_q$ induces inclusions of $H^*(\W_q)$ into $H^*(\D^r\W_q)$ and $H^*(\W_q)$ into $H^*(\D^r\W_q')$.
\end{proposition}
\begin{proof}
It is clear that the inclusion of $\W_q$ to $\D^\infty\W_q$ is a cochain map.
Suppose that $\omega\in\W_q$ is a cocycle and that $\omega=d\alpha+\beta$ holds for $\alpha\in\D^\infty\W_q$ and $\beta\in\mathscr{I}$.
Takeing the part of order zero, we have $\omega=d\alpha'+\beta'$, where $\alpha'$ and $\beta'$ are the terms of order zero of $\alpha$ and $\beta$, respectively.
Then, we have $\alpha'\in\W_q$ and the $\beta\in I$, where $I$ denotes the ideal of $\W_q$ by $\{\gamma\in\W_q\mid\norm{\gamma}>q\}$.
Hence $\omega$ represents the trivial element in $H^*(\W_q)$.
The proof also works for $\D^r\W_q'$.
\end{proof}

Thanks to Lemma~\ref{lem7.14}, characteristic classes for deformations of foliations are realized by means of $H^*(\D^r\W_q)$ as follows.

\begin{theorem}
\label{thm6.6}
\begin{enumerate}
\item[\textup{1)}]
The homomorphism $\widetilde\chi$ induces a mapping $\widehat\chi$ from $\widehat{D}^r(\CF)\times H^*(\D^r\W_q)$ to $H_{\mathrm{DR}}^*(M;\R)$.
The mapping $\widehat\chi$ is partially linear in the sense that if we fix $\widehat\alpha\in\widehat{D}^r(\CF)$, then $\widehat\chi_\alpha$ is a homomorphism.
Moreover, $\widehat\chi$ is independent of choice of connections and metrics.
On the other hand, $\widehat\chi$ depends on the homotopy type of trivialization of $Q(\CF)$.
\item[\textup{2)}]
If $f\colon\mathscr{D}^r(\CF)\to D^r(\CF)$ and $g\colon D^r(\CF)\to\widehat{D}^r(\CF)$ denote natural mappings, then we have $\widehat\chi_{g(\alpha)}(c)=\chi_{\alpha}(c)$ and $\left.\pdif{}{t}\chi_{f(a)}(c)\right|_{t=0}=\widecheck\chi_{a}(c)$ for $\alpha\in D^r(\CF)$ and $a\in\mathscr{D}^r(\CF)$.
\end{enumerate}
\end{theorem}
\begin{proof}
The characteristic classes are realized as characteristic classes for a foliation $\CF^{(r)}$.
Hence the proof reduces to the case of usual secondary classes.
See~\cite{Bott:LNM} for the part~1), and \cite{13},~\cite{asuke:GV} for the part~2).
The compatibility directly follows from the constructions, especially from Definition~\ref{def4.4}.
\end{proof}

\begin{remark}
The part~2) of Theorem~\ref{thm6.6} can be also seen as follows.
Given an element $\widehat\alpha=(\{{\overline\omega}^{(r)}\},\{{\overline\theta}^{(r)}\})$, we set $\widehat\omega=\sum_{k=0}^r\frac1{k!}t^k\overline\omega^{(k)}$ and $\widehat\theta=\sum_{k=0}^r\frac1{k!}t^k\overline\theta^{(k)}$.
Then, $\widehat\omega$ is a formal trivialization of $Q^*(\CF)$ in the sense that $\widehat\omega$ is a section to $Q^*(\CF)[t]$.
Similarly, $\widehat\theta$ is a formal Bott connection on $Q(\CF)$ in the sense that the formally defined connection by $\widehat\theta$ is valued in $Q(\CF)[t]$.
By using $\widehat\omega$ and $\widehat\theta$, we can apply arguments in the case of jets of deformations for the case of formal deformations.
If $r=\infty$, then it suffices to consider $Q^*(\CF)[[t]]$ and $Q(\CF)[[t]]$.
\end{remark}

\begin{definition}
We call $\widecheck\chi$, $\chi$ and $\widehat\chi$ the \textit{$r$-jets of characteristic mappings}.
If we clarify the order, they are denoted by $\widecheck\chi^r$, $\chi^r$, $\widehat\chi^r$, respectively.
\end{definition}

Since the characteristic classes for deformations are themselves characteristic classes for certain foliations, we have the following

\begin{theorem}
Let $\CF$ and $\CF'$ be foliations of $M$ and $M'$, respectively.
We assume that there is a smooth mapping $g\colon M\to M'$ such that $\CF=g^*\CF'$.
If $\chi$ and $\chi'$ denotes the $r$-jets of characteristic homomorphisms for $\CF$ and $\CF'$, then we have $\chi=g^*\chi'$.
More precisely, we have $\chi_{g^*\alpha}(c)=g^*\chi'_\alpha$.
\end{theorem}

\begin{corollary}
Let $\CF$ be a foliation of $M$, and $\alpha_0,\alpha_1$ infinitesimal deformations of order $r$.
If there is a foliation $\mathcal{F}'$ of $M\times[0,1]$ and infinitesimal deformation $\alpha'$ such that $\iota^*_i\alpha'=\alpha_i$ holds for $i=0,1$, where $\iota_i\colon M\to M\times\{i\}\subset M\times[0,1]$ is the inclusion, then $\chi_\alpha=\chi_{\alpha'}$.
\end{corollary}

\begin{remark}
Let $c\in H^*(\D^\infty\W_q)$.
If $c$ is represented by a cocycle which does not involve $h_{i,(0)}$, $1\leq i\leq q$, then $\chi(c)$ is independent of the choice of trivializations of $Q(\CF)$.
\end{remark}

Finally, we present some DGA's needed for study certain kinds of foliations.

\begin{definition}
\label{def7.20}
\begin{enumerate}
\item
If we consider real foliations without assuming the triviality of normal bundles, then we will make use of
\[
\D^r\WO_q=\left.\left(\Wedge[h_{2i'-1,(0)},h_{i,(a)}]_{1\leq i'\leq q',\substack{1\leq i\leq q\\ 1\leq a\leq r}}\otimes\R[c_{i,(b)}]_{\substack{1\leq i\leq q\\ 0\leq b\leq r}}\right)\right/\mathscr{I},
\]
where $q'$ is the largest integer such that $2q'-1\leq q$.
The ideal $\mathscr{I}$ is generated by elements $c\in\D^r\WO_q$ with $\norm{c}>q$, where $\norm{h_{2i'-1,(0)}}=0$, $\norm{h_{i,(a)}}=\max\{i-a,0\}$ if $a\geq1$ and $\norm{c_{i,(b)}}=\max\{i-b,0\}$.
\item
If we consider transversely holomorphic foliations, then the characteristic mappings are valued in $H^*(M;\C)$.
If the complex normal bundles of foliations under considerations are trivial, then we will make use of
\[
\D^r\W^\C_q=(\D^r\W_q\otimes\C)\wedge(\overline{\D^r\W_q\otimes\C}),
\]
where the homotopy type of trivialization is fixed when we consider the characteristic mapping.
We can make use of the holomorphic part of $\D^r\W_q^\C$, namely, we consider $\D^r\W_q\otimes\C$.
The latter is used in proofs of properties of $H^*(\D^r\W_q)$ in Theorem~\ref{thm4.37} and Proposition~\ref{prop10.14}.
\item
If we consider transversely holomorphic foliations without assuming the triviality of normal bundles, then we will make use of
\[
\D^r\WU_q=\left.\left(\Wedge[\widetilde{u}_{i,(0)},u_{i,(a)},\overline{u}_{i,(a)}]_{1\leq i\leq q,\ 1\leq a\leq r}\otimes\C[v_{i,(b)},\overline{v}_{i,(b)}]_{1\leq i\leq q,\ 0\leq b\leq r}\right)\right/\mathscr{I},
\]
where $d\widetilde{u}_{i,(0)}=v_{i,(0)}-\overline{v}_{i,(0)}$, $du_{i,(a)}=v_{i,(a)}$, $d\overline{u}_{i,(a)}=\overline{v}_{i,(a)}$ and $dv_{i,(b)}=d\overline{v}_{i,(b)}=0$ for $a\geq1$ and $b\geq0$.
We set $\norm{\widetilde{u}_{i,(0)}}=\norm{\widetilde{u}_{i,(0)}}'=0$, $\norm{u_{i,(a)}}=\norm{\overline{u}_{i,(a)}}'=\max\{i-a,0\}$, $\norm{u_{i,(a)}}'=\norm{\overline{u}_{i,(a)}}=0$ for $a\geq1$, and $\norm{v_{i,(b)}}=\norm{\overline{v}_{i,(b)}}'=\max\{i-b,0\}$, $\norm{v_{i,(b)}}'=\norm{\overline{v}_{i,(b)}}=0$ for $b\geq0$.
We define $\mathscr{I}$ to be the ideal generated by elements $c$ with $\norm{c}>q$ or $\norm{c}'>q$.
\item
Some of characteristic classes need assumption on foliations.
For example, the Godbillon--Vey class and the FLK class is defined for foliations with trivial canonical bundles, where the canonical bundle means the determinant bundle of the normal bundle.
In the real category, the triviality is not a strong constraint, while it is essential in the complex category.
The counterpart of the Godbillon--Vey class is called the Bott class.
Without assuming the triviality, only the imaginary part of the Bott class is defined with coefficients in $\R$ and the real part is a class with coefficients in $\R/\Z$.
In addition, the FLK class is no longer defined in such a case.
\end{enumerate}
\end{definition}

\begin{remark}
Some of the cohomology of DGA's listed in Definition~\ref{def7.20} are studied in \cite{asuke:UT}, where the Heitsch formula is not taken into account.
\end{remark}

\begin{remark}
Bases for $H^*(\W_q)$, $H^*(\WO_q)$ and $H^*(\W^\C_q)$, where $\W^\C_q=(\W_q\otimes\C)\wedge(\overline{\W_q\otimes\C})$, are well-known as the Vey bases~\cite{Godbillon}.
On the other hand, we do not know bases for $H^*(\D^r\W_q)$, etc. if $r\geq1$.
\end{remark}

\section{Variability and rigidity, comparison of $\chi$ and $\chi'$}

Following the standard definition, we introduce the following
\begin{definition}
\begin{enumerate}
\item
An element $\gamma$ of $H^*(\D^r\W_q)$ is said to be \textit{variable} if $\chi_{\widehat{\alpha}}(\delta\gamma)\neq0$ for some $\widehat\alpha\in\widehat{D}^r(\CF)$.
\item
An element $\gamma$ of $H^*(\D^r\W_q)$ is said to be \textit{rigid} if $\chi_{\widehat\alpha}(\delta\gamma)=0$ for any \mbox{$\widehat\alpha\in\widehat{D}^r(\CF)$}.
\end{enumerate}
\begin{enumerate}
\item[1')]
An element $\gamma$ of $H^*(\D^r\W_q)$ is said to be \textit{formally variable} if $\delta^k\gamma$ is nontrivial for any $k$.
\item[2')]
An element $\gamma$ of $H^*(\D^r\W_q)$ is said to be \textit{formally rigid} if $\delta\gamma=0$.
\end{enumerate}
\end{definition}

We have the following
\begin{lemma}
\label{lem7.22}
If $\gamma\in H^*(\D^r\W_q)$ is variable, then $\gamma$ is formally variable.
On the other hand, if $\gamma\in H^*(\D^r\W_q)$ is formally rigid, then $\gamma$ is rigid.
\end{lemma}
\begin{proof}
If $\gamma$ is variable, then there is an element $\widehat\alpha=(\{\overline\omega^{(k)}\},\{\overline\theta^{(k)}\})\in\widehat{D}^r(\CF)$ as in Definition~\ref{def3.9} such that $\chi_{\widehat\alpha}(\gamma)\neq0$.
For $l\in\N$, we set $\mu_l^{(k)}=\overline\omega^{(lk)}$ and $\xi_l^{(k)}=\overline\theta^{(lk)}$.
Then, $\widehat\alpha_l=(\{\mu_l^{(k)},\xi_l^{(k)}\})\in\widehat{D}^r(\CF)$ and we have $\chi_{\widehat\alpha_l}(\gamma)=l!\chi_{\widehat\alpha}(\CF)$, which is non-zero if $l\neq0$.
The second part follows from Theorem~\ref{thm6.6}.
\end{proof}

\begin{remark}
Suppose that $\widehat\alpha$ is given by a smooth family $\{\CF_t\}$ in the proof of Lemma~\ref{lem7.22}.
Then, $\widehat\alpha_l$ is given by the family $\{\CF_{t^l}\}$.
\end{remark}

The following is known.
The part 1) is essentially due to Thurston.
The part 2) is shown in \cite{asuke:FLK} (see also Proposition~\ref{prop10.14}).

\begin{theorem}
\label{thm4.37}
\begin{enumerate}
\item[\textup{1)}]
The Godbillon--Vey class $\GV\in H^{2q+1}(\D^r\W_q)$ is variable.
\item[\textup{2)}]
The Fuks--Lodder--Kotschick class $\FLK\in H^{2q+2}(\D^r\W_q)$ is variable.
\end{enumerate}
\end{theorem}

Let $\widetilde\W_q=\widetilde{\D^0\W_q}$ and $\W_q=\widetilde\W_q/\mathcal{I}_0$, where $\mathcal{I}_0$ is the ideal of $\widetilde\W_q$ generated by $\mathscr{I}_0'$ in Definition~\ref{relations}.
The Bott vanishing theorem implies that $\widetilde\chi$ restricted to $\widetilde\W_q$ induces a homomorphism $\chi_\CF\colon H^*(\W_q)\to H_{\mathrm{DR}}^*(M;\R)$.
It is known that $\chi_\CF$ does not depend on the choice of connections~\cite{Bott:LNM}.
On the other hand, we have a natural homomorphism $\rho^{q+1}_q\colon\W_{q+1}\to\W_q$ such that $\rho^{q+1}_q(h_i)=\begin{cases}
h_i, & i\neq q+1,\\
0, & i=q+1
\end{cases}$ and that $\rho^{q+1}_q(c_i)=\begin{cases}
c_i, & i\neq q+1,\\
0, & i=q+1
\end{cases}$, where $h_i=h_{i,0}$ and $c_i=c_{i,0}$.
The induced homomorphism on the cohomology is also denoted by $\rho^{q+1}_q$ by abuse of notations.

The following rigidity theorem is known.

\begin{theorem}[\cite{Heitsch:Topology}]
\label{thm7.12}
The image of $\rho^{q+1}_q$ consists of classes which are rigid under actual and infinitesimal deformations.
\end{theorem}

Theorem~\ref{thm7.12} states for a smooth family $\{\CF_t\}$ that $\left.\pdif{}{t}\chi_{\CF_t}(c)\right|_{t=0}=0$ holds if $c$ is in the image of $\rho^{q+1}_q$.
We will prove Theorem~\ref{thm7.12} in a generalized form as Theorem~\ref{thm7.25}.

There are some other results.
We present them in a reduced form.

\begin{theorem}[Bott~\cite{BB} for 1), Heitsch~\cite{14} for 1) and 2)]
\label{thm_BB}
Let $i,j_1,\ldots,j_q\in\N$ and assume that $i>0$, $i+j_1+2j_2+\cdots+j_q=q+1$.
\begin{enumerate}
\item
The class $h_ic_1^{j_1}\cdots c_q^{j_q}\in H^{2q+1}(\D^\infty\W_q)$ is variable.
\item
Suppose in addition that $i_2,\ldots,i_k\in\N$ and that $i<i_2<\cdots<i_k<q-1$.
Then, the class $h_ih_{i_2}\cdots h_{i_k}c_1^{j_1}\cdots c_q^{j_q}$ is variable.
\end{enumerate}
\end{theorem}

\begin{theorem}[Kamber--Tondeur~\cite{KT}, Hurder~\cite{Hurder:indep}, Asuke~\cite{asuke:indep}]
There exist examples of foliations of which some rigid classes in $H^*(\D^r\W_q)$ and $H^*(\D^r\WU_q)$ are non-trivial.
\end{theorem}

On the contrary, we have the following rigidity which is originally found in~\cite{Heitsch:Topology}.

\begin{definition}
We define a homomorphism $\widetilde\rho^{q+1}_q\colon\widetilde{\D^r\W_{q+1}}\to\widetilde{\D^r\W_q}$ by the condition that
\begin{align*}
&\widetilde\rho^{q+1}_q(c_{i,(l)})=\begin{cases}
c_{i,(l)}, & i\neq q+1,\\*
0, & i=q+1,
\end{cases}\\*
&\widetilde\rho^{q+1}_q(h_{i,(l)})=\begin{cases}
h_{i,(l)}, & i\neq q+1,\\*
0, & i=q+1.
\end{cases}
\end{align*}
The induced homomorphism from $H^*(\D^r\W_{q+1})$ to $H^*(\D^r\W_q)$ is denoted by $\rho^{q+1}_q$.
\end{definition}

\begin{example}
If $q=2$, then $h_2c_2\in H^7(\D^r\W_3)$ is in the image of $\rho^3_2$.
\end{example}

We will make use of the following variant of $\D^r\W_q$.

\begin{definition}
We set $\D^r\W_q^+=\widetilde{\D^r\W_q}/\mathcal{I}$, where $\mathcal{I}$ is the ideal of $\widetilde{\D^r\W_q}$ generated by $\mathscr{I}'_0$.
\end{definition}

Note that the natural inclusion of $H^*(\W_q)$ to $H^*(\D^r\W_q)$ factors through $H^*(\D^r\W_q^+)$.
Let $\rho^{q+1}_q{}^+$ be the homomorphism from $H^*(\D^r\W_{q+1}^+)$ to $H^*(\D^r\W_q^+)$ induced by~$\widetilde\rho^{q+1}_q$.

\begin{definition}
We set
\begin{align*}
&K_i(c_{j,(l)})=\begin{cases}
h_{i,(l+1)}, & j=i,\\*
0, & j\neq i
\end{cases}\\*
&K_i(h_{j,(l)})=0
\end{align*}
We extend $K_i$ to the whole $\widetilde{\D^r\W_q}$ as a (signed) derivation.
\end{definition}

\begin{lemma}
\label{lem7.10}
Let $\widetilde\delta_i$ be the time derivation which operates only $c_{i,(l)}$ and $h_{i,(l)}$.
Then, we have $\widetilde\delta_i\varphi=K_i(d\varphi)+d(K_i\varphi)$\/ for $\varphi\in\widetilde{\D^r\W_q}$.
\end{lemma}
\begin{proof}
We have $K_i(dc_{i,(l)})+dK_i(c_{i,(l)})=c_{i,(l+1)}$ and that $K_i(dh_{i,(l)})+dK_i(h_{i,(l)})=h_{i,(l+1)}$.
Let $\varphi\in\widetilde{\D^r\W_q}$ and suppose that $K_i(d\varphi)+dK_i(\varphi)=\widetilde\delta_i\varphi$.
We have
\begin{align*}
&\hphantom{{}={}}%
K_i(d(c_{i,(l)}\varphi))+dK_i(c_{i,(l)}\varphi)\\*%
&=K_i(c_{i,(l)}d\varphi)+d(h_{i,(l+1)}\varphi+c_{i,(l)}K(\varphi))\\*
&=h_{i,(l+1)}d\varphi+c_{i,(l)}K(d\varphi)+c_{i,(l+1)}\varphi-h_{i,(l+1)}d\varphi+c_{i,(l)}dK(\varphi)\\*
&=c_{i,(l+1)}\varphi+c_{i,(l)}\widetilde{\delta}_i\varphi\\*
&=\widetilde\delta_i(c_{i,(l)}\varphi)
\\
\intertext{and}
&\hphantom{{}={}}%
K_i(d(h_{i,(l)}\varphi))+dK_i(h_{i,(l)}\varphi)\\*%
&=K_i(c_{i,(l)}\varphi-h_{i,(l)}d\varphi)-d(h_{i,(l)}K(\varphi))\\*
&=h_{i,(l+1)}\varphi+c_{i,(l)}K(\varphi)+h_{i,(l)}K(d\varphi)-c_{i,(l)}K(\varphi)+h_{i,(l)}dK(\varphi)\\*
&=h_{i,(l+1)}\varphi+h_{i,(l)}\widetilde{\delta}_i\varphi\\*
&=\widetilde\delta_i(h_{i,(l)}\varphi).
\end{align*}
Thus we are done.
\end{proof}

\begin{definition}
We set $K=\sum_{i=1}^qK_i$.
\end{definition}

\begin{lemma}
We have $\norm{K(\varphi)}\geq\norm{\varphi}-1$.
\end{lemma}

By Lemma~\ref{lem7.10}, we have
\begin{lemma}
\label{lem7.8}
We have $\widetilde{\delta}\varphi=K(d\varphi)+d(K\varphi)$ on $\widetilde{\D^r\W_q}$.
\end{lemma}

\begin{remark}
The mapping $K$ does \textit{not} induce a cochain homotopy on $\D^r\W_q$.
Indeed, $K$ does not preserve $\mathscr{I}$.
Actually, $K$ does not preserve $\mathscr{I}'$, either.
\end{remark}

\begin{theorem}
\label{thm7.25}
Let $\rho^{q+1}_q{}'$ be the composite of $\rho^{q+1}_q{}^+$ and the natural homomorphism from $H^*(\D^r\W_q^+)$ to $H^*(\D^r\W_q)$.
Then the image of $\rho^{q+1}_q{}'$ consists of formally rigid elements.
\end{theorem}
\begin{proof}
If $\gamma\in H^*(\D^r\W_q)$ is in the image of $\rho^{q+1}_q{}'$, then there is a cocycle in $\varphi\in\D^r\W_{q+1}^+$ which represents $\gamma$ as an element of $\D^r\W_q$.
Let $\widetilde\varphi$ be a representative of $\varphi$ in $\widetilde{\D^r\W_{q+1}}$.
Then, $\delta\gamma$ is represented by $\widetilde\rho^{q+1}_q\widetilde\delta\widetilde\varphi$.
On the other hand, we have $\widetilde\delta\widetilde\varphi=K(d\widetilde\varphi)+d(K\widetilde\varphi)$ by Lemma~\ref{lem7.8}.
As $\varphi$ is a cocycle in $\D^r\W_{q+1}^+$, we have $\norm{d\widetilde\varphi}>q+1$.
Therefore, $\delta\gamma$ is trivial.
\end{proof}

\begin{corollary}[Heitsch \cite{Heitsch:Topology}]
The image of $H^*(\W_{q+1})$ in $H^*(\D^r\W_q)$ are formally rigid.
\end{corollary}

\begin{corollary}
Let $\widehat\chi\colon\widehat{D}^r(\CF)\times H^*(\D^r\W_q)\to H_{\mathrm{DR}}^*(M;\R)$ be the characteristic homomorphism.
If $k\geq1$, then $\widehat\chi_{\widehat\alpha}\circ\delta^k$ is trivial on the image of $\rho^{q+1}_q{}'$ for any $\widehat\alpha\in\widehat{D}^r(\CF)$.
\end{corollary}
\begin{proof}
Let $\gamma\in H^*(\D^r\W_q)$ be in the image of $\rho^{q+1}_q{}'$.
Then, there is a cocycle in $\varphi\in\D^r\W_{q+1}$ which represents $\gamma$ as an element of $\D^r\W_q$.
We have $\norm{d\varphi}>q+1$.
On the other hand, by Lemma~\ref{lem7.10}, $\widehat\chi_{\widehat\alpha}(\delta(\gamma))$ is represented by $K(d\varphi)+d(K\varphi)$.
By Theorem~\ref{thm6.9}, we have $\norm{K(d\varphi)}>q$ so that $\delta\chi_{\widehat\alpha}(\gamma)$ is trivial.
\end{proof}

Characteristic classes for deformations of are studied in~\cite{asuke:UT}, where $r=1$ and the Heitsch formula is \textit{not} taken into account.
Namely, the classes studied there are those defined in terms of $H^*(\D^1\W'_q)$.
Combining with Theorem~\ref{thm3.17}, we obtain the following

\begin{theorem}
If $c\in H^*(\D^q\W'_q)$, then $\overline{c}$ denotes the image of $c$ under the natural mapping from $H^*(\D^1\W'_q)$ to $H^*(\D^1\W_q)$.
If $\widehat\alpha\in\widehat{D}^1(\mathcal{F})$, then, we have $\chi_{\widehat\alpha}(c)=\chi_{\widehat\alpha}(\overline{c})$, where the element of $H^1(M;\Theta_\CF)$ induced by $\widehat\alpha$ is again denoted by $\widehat\alpha$ by abuse of notations.
\end{theorem}

\begin{remark}
The arguments developed so far are valid also for $\D^r\WO_q$, $\D^r\W^\C_q$ and $\D^r\WU_q$ with suitable adaptations.
\end{remark}

\begin{remark}
We also have some elements analogous to the FLK class if $q\geq2$.
For example, if we set $c=h_2\delta(h_1c_2)=h_2\dot{h}_1c_2-h_1h_2\dot{c}_2$, then $c\in H^8(\D^\infty\W_2)$ is formally variable and $\sigma(c)=0$.
We do not know if $c$ is variable or not.
\end{remark}

\section{Structure of $H^*(\D^\infty\W_q)$}

Recall that we have derivations $\widetilde\delta$ and $\widetilde\sigma$ on $\widetilde{\D^\infty\W_q}$.
As these derivations leave $\mathscr{I}$ invariant and commute with the differential $d$, they induce derivations on $H^*(\D^\infty\W_q)$, which are denoted by $\delta$ and $\sigma$, respectively.


If $T$ is a linear operator which acts on a vector space $V$, then we refer to the eigenspace of $T$ of eigenvalue $\lambda$ as the $\lambda$-eigenspace of $T$.
\begin{definition}
Let $\lambda\in\R$ and $k\in\N$.
Let $\widetilde{F}_\lambda$ be the $\lambda$-eigenspace of $\widetilde\delta\circ\widetilde\sigma$.
The subspace of $\widetilde{F}_\lambda$ which consists of order $k$ is denoted by $\widetilde{F}_{\lambda,k}$.
Similarly, $\widetilde{F}'_\lambda$ denotes the $\lambda$-eigenspace of $\widetilde\sigma\circ\widetilde\delta$ and $\widetilde{F}'_{\lambda,k}$ denotes the subspace which consists of elements of order $k$.
\end{definition}
We have $\widetilde{F}_\lambda=\bigoplus_{k\in\N}\widetilde{F}_{\lambda,k}$ and $\widetilde{F}'_\lambda=\bigoplus_{k\in\N}\widetilde{F}'_{\lambda,k}$.

\begin{lemma}
\label{eigenspaces}
\begin{enumerate}[label=\textup{\arabic*)}]
\item
We have $\widetilde{F}_{\lambda,k}=\widetilde{F}'_{\lambda+k,k}$.
\item
Let $\widetilde\sigma_{\lambda,k}$ be the restriction of $\widetilde\sigma$ to $\widetilde{F}_{\lambda,k}$.
Then, the image of $\widetilde\sigma_{\lambda,k}$ is contained in $\widetilde{F}_{\lambda-k+1,k-1}$.
If $\lambda\neq0$, then $\widetilde\sigma_{\lambda,k}$ is an isomorphism and the inverse is given by $\frac1{\lambda}\widetilde\delta$.
\end{enumerate}
\end{lemma}
\begin{proof}
First we show~1).
Suppose that $\omega$ is of order $k$.
We have $\widetilde\sigma\circ\widetilde\delta(\omega)=k\omega+\widetilde\delta\circ\widetilde\sigma(\omega)$ by Lemma~\ref{lem_new6.2}.
Hence $\widetilde\sigma\circ\widetilde\delta(\omega)=(\lambda+k)\omega$ if and only if $\omega\in\widetilde{F}_{\lambda,k}$.

Next, we show~2).
If $\omega\in\widetilde{F}_{\lambda,k}$, then we have $\widetilde\delta\circ\widetilde\sigma(\widetilde\sigma(\omega))=\widetilde\sigma\circ\widetilde\delta(\widetilde\sigma(\omega))-\mbox{$(k-1)$}\widetilde\sigma(\omega)=(\lambda-k+1)\widetilde\sigma(\omega)$ so that we have $\widetilde\sigma(\omega)\in\widetilde{F}_{\lambda-k+1,k-1}$.
If conversely $\omega'\in\widetilde{F}_{\lambda-k+1,k-1}$, then we have $\widetilde\sigma\circ\widetilde\delta(\omega')=\lambda\omega'$ by~1).
As we have $\widetilde\delta\circ\widetilde\sigma(\omega)=\lambda\omega$ on $\widetilde{F}_{\lambda,k}$, the claim holds.
\end{proof}

In order to describe $\widetilde{\D^\infty\W_q}$, we introduce a sequence as follows.
\begin{definition}
\label{def9.3}
If $k\in\N$ and if $k\geq1$, then we define a sequence $\{\lambda_{m,k}\}_{1\leq m\leq k}$ by setting $\lambda_{m,k}=\sum_{m'=1}^{k-1}(k-m')=\frac12(m-1)(2k-m)$.
We set $\lambda_{0,0}=0$.
\end{definition}

The following is immediate.
\begin{lemma}
\begin{enumerate}
\item[\textup{1)}]
We have $\lambda_{1,k}=0$ for any $k\geq1$ and that $\lambda_{m,k}\geq1$ if $m\geq2$.
\item[\textup{2)}]
We have $\lambda_{m,k}+k=\lambda_{m+1,k+1}$.
\end{enumerate}
\end{lemma}

\begin{proposition}
\label{prop9.5}
\begin{enumerate}
\item[\textup{1)}]
We have $\widetilde{\D^\infty\W_q}{}^0=\widetilde{F}_{0,0}$.
If $k\geq1$, then we have\linebreak \mbox{$\widetilde{\D^\infty\W_q}{}^k=\bigoplus_{m=1}^k\widetilde{F}_{\lambda_{m,k},k}$}.
\item[\textup{2)}]
The mapping $\widetilde\sigma$ is an isomorphism to the image when restricted to $\bigoplus_{\substack{n\geq1\\ k\geq2}}\widetilde{F}_{n,k}$.
On the other hand, the mapping $\widetilde\delta$ is an isomorphism to the image when restricted to $\bigoplus_{\substack{k\geq1\\1\leq m\leq k}}\widetilde{F}_{\lambda_m,k}=\bigoplus_{k\geq1}\widetilde{\D^\infty\W_q}{}^k$.
\end{enumerate}
\end{proposition}
\begin{proof}
If $k=0$ or $k=1$, then we have $\widetilde\sigma=0$ on $\widetilde{\D^\infty\W_q}{}^k$.
Hence $\widetilde{\D^\infty\W_q}{}^k=\widetilde{F}_{0,k}$.
Assume that 2) holds for a $k\geq1$ and that $\omega\in\widetilde{\D^\infty\W_q}{}^{k+1}$.
Then, $\widetilde\sigma\omega\in\widetilde{\D^\infty\W_q}{}^k$ so that there exist $\mu_m\in\widetilde{F}_{\lambda_{m,k},k}$, $1\leq m\leq k$, such that $\widetilde\sigma(\omega)=\mu_1+\cdots+\mu_k$.
Noticing that $\widetilde\delta(\mu_m)\in\widetilde{F}_{\lambda_{k,m}+k,k+1}=\widetilde{F}_{\lambda_{m+1,k+1},k+1}$, we set $\nu_m=\frac1{\lambda_{m+1,k+1}}\widetilde\delta(\mu_m)$ for $1\leq m\leq k$ and $\nu_0=\omega-\sum_{m=1}^k\nu_m$.
We have
\begin{align*}
\widetilde\delta\circ\widetilde\sigma(\nu_0)&=\widetilde\delta\circ\widetilde\sigma(\omega)-\sum_{m=1}^k\widetilde\delta\circ\widetilde\sigma(\nu_m)\\*
&=\widetilde\delta\circ\widetilde\sigma(\omega)-\sum_{m=1}^k\lambda_{m+1,k+1}\nu_m\\*
&=\widetilde\delta\circ\left(\widetilde\sigma(\omega)-\sum_{m=1}^k\mu_m\right)\\*
&=0.
\end{align*}
Therefore, $\nu_0\in\widetilde{F}_{\lambda_{1,k+1},k+1}$ and we have $\widetilde{\D^\infty\W_q}{}^{k+1}=\bigoplus_{m=1}^{k+1}\widetilde{F}_{\lambda_{m,k+1},k+1}$.
Hence~1) holds.
On the other hand, Lemma~\ref{eigenspaces} shows that $\widetilde\delta$ isomorphically maps $\widetilde{F}_{\lambda_{m,k},k}$ to $\widetilde{F}_{\lambda_{m+1,k+1},k+1}$ if $k\geq1$.
Hence the claim~2) follows from~1).
\end{proof}

\begin{corollary}
\label{cor9.6}
Let $k\geq1$.
If $p_{m,k}$ denotes the projector from $\widetilde{\D^\infty\W_q}{}^k$ to $\widetilde{F}_{\lambda_{m,k},k}$, then we have
\begin{align*}
&p_{1,k}=\sum_{i=0}^{k-1}(-1)^i\frac{2^i(2k-i-2)!}{(2k-2)!i!}\widetilde\delta^i\circ\widetilde\sigma^i,\\*
&p_{k-a,k}=\frac{2^{k-a-1}(2a+1)!}{(k-a-1)!(k+a)!}\widetilde\delta^{k-a-1}\circ p_{1,k}\circ\widetilde\sigma^{k-a-1},
\end{align*}
where $0\leq a\leq k-1$.
\end{corollary}
\begin{proof}
First, if $k\geq2$ and if $1\leq a\leq k-1$, then we have $p_{k-a,k}=\frac1{\lambda_{k-a,k}}\widetilde\delta\circ p_{k-a-1,k-1}\circ\widetilde\sigma$.
The second formula follows from this.

Next, we have
\begin{align*}
p_{1,k}&=\id-\sum_{i=1}^{k-1}\frac1{\lambda_{i+1,k}}\widetilde\delta\circ p_{i,k-1}\circ\widetilde\sigma\\*
&=\id-\sum_{i=1}^{k-1}\frac1{\lambda_{i+1,k}\cdots\lambda_{2,k-i+1}}\widetilde\delta^i\circ p_{1,k-i}\circ\widetilde\sigma^i\\*
&=\id-\sum_{i=1}^{k-1}\frac{2^i(2k-2i-1)!}{i!(2k-i-1)!}\widetilde\delta^i\circ p_{1,k-i}\circ\widetilde\sigma^i.
\end{align*}
Note that $p_{1,1}=\id$ because $\widetilde{\D^\infty\W_q}{}^1=\widetilde{F}_{0,1}$.
Therefore, the first formula in the statement holds if $k=1$.
Assume that the formula is valid for $p_{1,l}$ with $l<k$.
Then, we have
\begin{align*}
p_{1,k}&=\id-\sum_{i=1}^{k-1}\sum_{l=0}^{k-i-1}\frac{2^i(2k-2i-1)!}{i!(2k-i-1)!}(-1)^l\frac{2^l(2k-2i-l-2)!}{(2k-2i-2)!l!}\widetilde\delta^{i+l}\circ\widetilde\sigma^{i+1}\\*
&=\id-\sum_{i=1}^{k-1}\sum_{m=i}^{k-1}(-1)^{m-i}\frac{2^m(2k-2i-1)!(2k-m-i-2)!}{i!(m-i)!(2k-i-1)!(2k-2i-2)!}\widetilde\delta^m\circ\widetilde\sigma^m\\*
&=\id-\sum_{m=1}^{k-1}\sum_{i=1}^m(-1)^{m-i}\frac{2^m(2k-2i-1)!(2k-m-i-2)!}{i!(m-i)!(2k-i-1)!(2k-2i-2)!}\widetilde\delta^m\circ\widetilde\sigma^m,
\end{align*}
where $m=i+l$.

Now let $P(i)=\frac1{(2k-i-1)\cdots(2k-2i-2)(2k-2i)\cdots(2k-m-i-1)}$ for $1\leq i\leq m$.
Then, we have
\[
p_{1,k}=\id-\frac{2^m}{m!}\sum_{m=1}^{k-1}\sum_{i=1}^m(-1)^{m-i}\dbinom{m}{i}P(i)\widetilde\delta^m\circ\widetilde\sigma^m,
\]
Hence, it suffices to show that
\[
\stepcounter{theorem}
\tag{\thetheorem}
\label{eq8.7_}
\frac{2^m}{m!}\sum_{i=1}^m(-1)^{i-1}\dbinom{m}{i}P(i)=\frac{2^m(2k-m-2)!}{m!(2k-2)!}
\]
holds for $1\leq m\leq k-1$.
Actually, we have
\begin{align*}
\stepcounter{theorem}
\tag{\thetheorem}
\label{eq8.7}
&\hphantom{{}={}}%
\sum_{i=1}^m(-1)^{i-1}\dbinom{m}{i}P(i)\\*
&=\sum_{i=1}^{a-1}(-1)^{i-1}\dbinom{m}{i}P(i)+(-1)^{a-1}\dbinom{m-1}{a-1}\frac1{(2k-a-1)\cdots(2k-m-a)}.
\end{align*}
Indeed, the equality~\eqref{eq8.7} holds if $a=m$.
On the other hand, we have
\begin{align*}
&\hphantom{{}={}}%
\frac1{2k-2a+1}\dbinom{m-1}{a-1}-\frac1{2k-a}\dbinom{m}{a-1}\\*
&=\frac{(m-1)!}{(a-1)!(m-a+1)!}\frac{(m-a+1)(2k-a)-m(2k-2a+1)}{(2k-2a+1)(2k-a)}\\*
&=-\frac{(m-1)!}{(a-2)!(m-a+1)!}\frac{2k-m-a}{(2k-2a+1)(2k-a)}\\*
&=-\dbinom{m-1}{a-2}\frac{2k-m-a}{(2k-2a+1)(2k-a)}.
\end{align*}
It follows that the equality~\eqref{eq8.7} also holds if $a=m-1$.
Hence, by induction, we have $\sum_{i=1}^m(-1)^{i-1}\dbinom{m}{i}P(i)=\dbinom{m-1}{a-1}\frac1{(2k-2)\cdots(2k-m-1)}$.
Thus the equality~\eqref{eq8.7_} holds so that we are done.
\end{proof}

As the above arguments concern linear spaces, we can apply to $\delta$ and $\sigma$ on $H^*(\D^\infty\W_q)$.

\begin{definition}
Let $H^*(\D^\infty\W_q)^k$ be the subspace of $H^*(\D^\infty\W_q)$ which is spanned by elements of order $k$.
The $\lambda$-eigenspace of $\delta\circ\sigma$ is denoted by $F_\lambda$.
We set $F_{\lambda,k}=F_\lambda\cap H^*(\D^\infty\W_q)^k$.
\end{definition}

\begin{theorem}
\label{structure}
\begin{enumerate}
\item[\textup{1)}]
We have $H^*(\D^\infty\W_q)^0=F_{0,0}$.
If $k\geq1$, then we have $H^*(\D^\infty\W_q)^k=\bigoplus_{m=1}^kF_{\lambda_{m,k},k}$, where $\lambda_{m,k}$ is given in Definition~\ref{def9.3}.
\item[\textup{2)}]
The derivation $\sigma$ is an isomorphism to the image when restricted to $\bigoplus_{\substack{n\geq1\\ k\geq2}}F_{n,k}$.
On the other hand, the derivation $\delta$ is an isomorphism to the image when restricted to $\bigoplus_{k\geq1}H^*(\D^\infty\W_q)^k$.
\item[\textup{3)}]
Let $k\geq1$.
If $p_{m,k}$ denotes the projector from $H^*(\D^\infty\W_q)^k$ to $F_{\lambda_{m,k},k}$, then we have
\begin{align*}
&p_{1,k}=\sum_{i=0}^{k-1}(-1)^i\frac{2^i(2k-i-2)!}{(2k-2)!i!}\delta^i\circ\sigma^i,\\*
&p_{k-a,k}=\frac{2^{k-a-1}(2a+1)!}{(k-a-1)!(k+a)!}\delta^{k-a-1}\circ p_{1,k}\circ\sigma^{k-a-1},
\end{align*}
where $0\leq a\leq k-1$.
\end{enumerate}
\end{theorem}

Theorem~\ref{structure} implies that $H^*(\D^\infty\W_q)$ is completely determined by $\bigoplus_{k=0}^{\infty}F_{0,k}$.

As the projectors commute with $d$, we have the following
\begin{lemma}
\begin{enumerate}
\item[\textup{1)}]
If $E_{\lambda,k}\subset\D^\infty\W_q$ denotes the $\lambda$-eigenspace of $\delta\circ\sigma$, then $E_{\lambda,k}$ is closed under $d$.
\item[\textup{2)}]
The eigenspace $F_{\lambda,k}$ is the cohomology of $(E_{\lambda,k},d)$.
\end{enumerate}
\end{lemma}

The above arguments also work for $H^*(\D^\infty\W_q')$ with $\delta'$ and $\sigma'$ after slight adaptations.
We sketch the arguments and results.

First, we have the following

\begin{lemma}
Let $c\in\D^r\W_q'$ and $c=\sum_{l=1}^mc_l$ be the decomposition of $c$ according to the length.
Then, $c$ is closed if and only if $c_l$ is closed for any $l$.
Similarly, $c$ is exact if and only if $c_l$ is exact for any $l$.
\end{lemma}

It is easy to see that the length is well-defined on $\D^r\W_q$, $\D^r\W_q'$, $H^*(\D^r\W_q)$ and $H^*(\D^r\W_q')$.
The subspace of $H^*(\D^r\W_q')^k$ naturally defined by $\widetilde{L}_l$ is denoted by $H^*(\D^r\W_q')^{k,l}$. where $\widetilde{L}_l$ is the subspace of $\widetilde{\D^r\W_q}$ defined in Definition~\ref{def5.8}.
We have $H^*(\D^r\W_q')^k=\bigoplus_{l=0}^{\infty}H^*(\D^r\W_q')^{k,l}$.

Let $L_{\lambda,k,l}$ be the eigenspace of $\delta'\circ\sigma'$ on $H^*(\D^r\W_q')^{k,l}$.

\begin{lemma}[cf. Lemma~\ref{eigenspaces}]
If $c\in L_{\lambda,k,l}$, then $\sigma(c)\in L_{\lambda-l,k-1,l}$ and $\delta(c)\in L_{\lambda+l,k+1,l}$.
If $\lambda\neq0$, then $\sigma'$ is an isomorphism and the inverse is given by $\frac1{\lambda}\delta$.
\end{lemma}

\begin{proposition}
We have $H^*(\D^r\W_q')^{k,l}=\bigoplus_{i=0}^kL_{il,k,l}$.
If $p'_{i,k,l}$ denotes the projector from $H^*(\D^r\W_q')^{k,l}$ to $L_{il,k,l}$, then we have
\[
p'_{i,k,l}(c)=\sum_{m=0}^{k-i}\frac{(-1)^m}{i!\,m!\,l^{m+i}}\delta^{m+i}\circ\sigma^{m+i}(c).
\]
\end{proposition}

From these Lemmata and Proposition, we can deduce the following

\begin{theorem}
\label{thm9.11}
The derivation $\delta'$ is injective on $H^*(\D^\infty\W_q')$.
\end{theorem}
The difference of the domain of injectivity in Theorem~\ref{thm9.11} and Theorem~\ref{structure} is derived from Lemma~\ref{lem6.9} and Lemma~\ref{lem_new6.2}.

If $q=1$, then we have $\D^r\W_1=\D^r\W_1'$ so that we have the following

\begin{corollary}
\label{cor9.16}
The derivation $\delta$ is injective on $H^*(\D^\infty\W_1)$.
\end{corollary}

\begin{corollary}
\label{cor8.8}
An element $\gamma\in H^*(\D^\infty\W_q)$ is either formally rigid or formally variable.
Precisely, a formally rigid element is an element of $F_{0,0}\subset H^*(\D^\infty\W_q)^0=H^*(\W_q)$.
On the other hand, elements of \/$\bigoplus_{k=1}^{\infty}H^*(\D^\infty\W_q)^k$ are formally variable.
If $q=1$, then any element of $H^*(\D^\infty\W_1)$ is formally variable.
\end{corollary}
\begin{proof}
This is because $\delta$ is injective on $\bigoplus_{k=1}^{\infty}H^*(\D^\infty\W_q)^k$.
If in addition $q=1$, then $\delta$ is injective.
\end{proof}

\begin{remark}
We cannot tell at once if an element of $H^*(\W_q)$ is either formally variable or rigid unless it is covered by the Heitsch theorem (Theorem~\ref{thm7.25}).
\end{remark}

\begin{example}
We have $\GV_q\in F_{0,0}\subset H^*(\D^\infty\W_q)$.
The formal variability of $\GV_1$ follows from Corollary~\ref{cor8.8}.
If $q\geq2$, then the formal variability follows from the existence of examples (Theorem~\ref{thm4.37}).
On the other hand, we have $\FLK_q\in F_{0,1}$ so that it is a priori formally variable.
\end{example}

\begin{remark}
It is a question to find a set of elements $\{\gamma_i\}$ of $H^*(\D^\infty\W_q)$ which generates the whole space under algebraic operations and $\delta$.
By Corollary~\ref{cor8.8}, it suffices to look at $\bigoplus_{k\in\N}F_{0,k}$ for this purpose.
We will study $F_{0,k}$ for small $k$ when $q=1$ in the next section (Theorem~\ref{thm10.1}, Proposition~\ref{prop10.6} and Theorem~\ref{thm10.7}).
\end{remark}

\begin{remark}
Arguments in this section are largely applicable for $\D^r\WO_q$.
On the other hand, if we study $\D^r\W^\C_q$ or $\D^r\WU_q$, we need to consider the integration $\sigma$ separately on the holomorphic part and on the anti-holomorphic part.
It makes the study especially difficult for $H^*(\D^r\WU_q)$.
\end{remark}

\section{A partial description of $H^*(\D^\infty\W_1)$}
If $q=1$, then we have additional information on $H^*(\D^\infty\W_1)$.
In this section, $c_{1,(i)}$ and $h_{1,(j)}$ are denoted by $c_{(i)}$ and $h_{(j)}$, respectively.

Note that we have $\mathscr{I}=\mathscr{I}'$.
In particular, $\mathscr{I}$ is generated by $\{\delta^k(c^2)\}_{k\in\N}$.

The following is easy, nevertheless it is fundamental.

\begin{theorem}
\label{thm10.1}
We have $F_{0,0}\cong\langle\GV\rangle$ and $F_{0,1}\cong\langle\FLK\rangle$, where $\GV=h_{(0)}c_{(0)}$ and $\FLK=h_{(0)}h_{(1)}c_{(0)}$.
\end{theorem}

\begin{definition}
An element $\omega\in\D^\infty\W_1$ is said to be of \textit{type $(a,b)$} if $\omega$ is a linear combination of elements of the form $h_{(i_1)}\cdots h_{(i_a)}c_{(j_1)}\cdots c_{(j_b)}$.
\end{definition}

The type determines a filtration of $\D^\infty\W_1$.
It is easy to see that this filtration induces a one of $H^*(\D^\infty\W_1)$ and that is compatible with $\delta$ and $\sigma$.

The Godbillon--Vey class is distinguished in the following sense.

\begin{theorem}
\label{thm10.3}
The only non-trivial classes of type $(1,b)$ are the Godbillon--Vey class and its derivatives which are of type $(1,1)$.
\end{theorem}
\begin{proof}
Let $\widetilde\omega\in\widetilde{\D^\infty\W_1}$ represent a class $\gamma\in H^*(\D^\infty\W_1)$ of type $(1,b)$.
By repeatedly applying $\widetilde\sigma$, we may assume that $\widetilde\sigma\omega\in\mathscr{I}$.
If $b=0$, then $\widetilde\omega$ cannot be closed in $\D^\infty\W_1$ unless $\widetilde\omega=0$.
Suppose that $b\geq2$.
We define $L\colon\widetilde{\D^\infty\W_1}\to\widetilde{\D^\infty\W_1}$ as follows.
We first~set
\begin{align*}
&L(c_{(i)})=h_{(i)},\\*
&L(h_{(i)})=0.
\end{align*}
Then, we extend $L$ as a signed derivation to the whole $\widetilde{\D^\infty\W_1}$.
A similar argument as in the proof of Lemma~\ref{lem7.10} shows that we have $(1+b)\widetilde\omega=L(d\widetilde\omega)+d(L\widetilde\omega)$.
Suppose now that $d\widetilde\omega\in\mathscr{I}$.
If we set $R_i=\widetilde{\delta}^i(c_{(0)}{}^2)$, then we have $d\widetilde\omega=\sum_{i=0}^m\alpha_iR_i$ for some $m$.
As $\widetilde\omega$ is of type $(1,b)$, each $\alpha_i$ is of type $(0,b-1)$.
Hence we can find $\beta_i\in\widetilde{\D^\infty\W_1}$ such that $d\beta_i=\alpha_i$.
We now set $\widetilde\mu=\sum_{i=0}^m\beta_iR_i$ and $\widetilde\nu=\frac1{1+b}L(\widetilde\omega-\widetilde\mu)$.
As $d(\widetilde\omega-\widetilde\nu)=0$, we have $\widetilde\omega-\widetilde\mu=d\widetilde\nu$.
Hence $\gamma$ is also represented by $\widetilde\mu$ so that $\gamma$ is trivial.

Finally we assume that $b=1$.
In this case, we may assume that $\ord\omega=k$.
Then, we have $\widetilde\omega=\sum_{i=0}^l\alpha_ih_{(i)}c_{(r-i)}$ for some $\alpha_i\in\R$.
If $r\leq1$, then the theorem follows from Theorem~\ref{thm10.1} so that we assume that $r\geq2$.
If $r$ is even, then repeatedly applying $\widetilde\delta$ to $\widetilde\omega$, we see that $\alpha_i=0$ for $i\neq\frac{k}2$.
Hence we have $\widetilde\omega=\alpha_jh_{(j)}c_{(j)}$, where $j=\frac{k}2$.
As $j\geq1$, $\widetilde\omega\in\mathscr{I}$ if and only if $\alpha_j=0$.
If $r$ is odd, then we see that $\alpha_i=0$ unless $i=j$ or $i=j+1$, where $j=\frac{k-1}2$.
In addition, we see that $\alpha_j+\alpha_{j+1}=0$.
It follows that $\widetilde\omega=\alpha_j(h_{(j)}c_{(j+1)}-h_{(j+1)}c_{(j)})=-\alpha_jd(h_{(j)}h_{(j+1)})$.
Therefore, $\widetilde\omega$ represents the trivial class.
\end{proof}

\begin{remark}
The latter half of the proof of Theorem~\ref{thm10.3} is a kind of repetition of a proof of triviality of $H^*(\widetilde{\W_q})$.
In addition, we can slightly reduce the last part of the proof of Theorem~\ref{thm10.3} by using $\sigma'$ instead of $\sigma$.
\end{remark}

In what follows, we will study $F_{0,k}$ for $2\leq k\leq 5$, without restricting the type.

First, we show the following

\begin{lemma}
\label{lem10.8}
Elements of the form $c_{(i)}c_{(j)}c_{(k)}$ with $i+j+k\leq5$ are trivial in~$\D^r\W_1$ if $r\geq5$.
\end{lemma}
\begin{proof}
Let $\omega_{ijk}=c_{(i)}c_{(j)}c_{(k)}$.
We may assume that $i\leq j\leq k$.
First note that we have $c_{(0)}{}^2=c_{(0)}c_{(1)}=0$.
In general, $c_{(i)}c_{(j)}$, $i\leq j$, is a linear combination of elements of the form $c_{(k)}c_{(l)}$, $k\leq l$, with $k+l=i+j$ and $(k,l)\neq(i,j)$.

If $i=0$, then $\omega_{0jk}=0$ unless $j\geq2$.
The element $\omega_{022}$ is a linear combination of $\omega_{013}$ and $\omega_{004}$ which are both trivial.
The element $\omega_{023}$ is a linear combination of $\omega_{014}$ and $\omega_{005}$ which are also trivial.
Suppose that $i=1$.
We have $\omega_{111}=-\omega_{021}$ and $\omega_{113}=-\omega_{023}$ so that they are trivial.
The element $\omega_{112}$ is proportional to $\omega_{103}$ so that it is trivial.
The element $\omega_{122}$ is proportional to $\omega_{032}$ so that it is also trivial.
Thus we are done.
\end{proof}

\begin{proposition}
\label{prop10.6}
We have $F_{0,k}\subset H^*(\D^\infty\W_1)\cong\{0\}$\, for $k=2,3,4$.
\end{proposition}
\begin{proof}
For a while, $c,\dot{c},\ddot{c},\dddot{c}$ denote $c_{1,(0)},c_{1,(1)},c_{1,(2)},c_{1,(3)}$, respectively.
We make use a similar notations for $h$.
Recall that $p_{1,k}$ denotes the projection from $H^*(\D^\infty\W_1)^k$ to~$F_{0,k}$.

Let $k=2$.
As $c^2=c\dot{c}=0$, the following elements gives rise to a set of basis for $(\D^\infty\W_1)^2$:
\[
\ddot{h}, h\ddot{h}, \ddot{h}c, h\ddot{h}c,%
\ddot{c}, \ddot{c}h, \ddot{c}c, \ddot{c}hc,%
\dot{h}\dot{c}, h\dot{h}\dot{c},%
\dot{c}^2, \dot{c}^2h.
\]
The images of these elements by $p_{1,2}$ to the $0$-eigenspace of $\delta\sigma$ are equal~to
\[
0, 0, -\dot{h}\dot{c},-h\dot{h}\dot{c},%
0, -\dot{c}\dot{h}, \ddot{c}c, \ddot{c}hc,%
\dot{h}\dot{c},h\dot{h}\dot{c},%
\dot{c}^2,\dot{c}^2h.
\]
Hence cocycles in the $0$-eigenspace $F_{0,2}$ of $\delta\sigma$ are linear combinations of
\[
\dot{c}^2,\ddot{c}c,\dot{c}^2h,\ddot{c}hc.
\]
As we have $\dot{c}^2h=d(\dot{c}\dot{h}h)$ and $\ddot{c}hc=\ddot{h}hc$, all of these elements are exact.

Next, we study $H^*(\D^\infty\W_1)^3$.
We have
\begin{align*}
&p_{1,3}(\dddot{h})=0,\\*
&p_{1,3}(\dddot{h}c)=-\frac12\ddot{h}\dot{c}+\frac12\dot{h}\ddot{c}=-\frac12d(\dot{h}\ddot{h}),\\*
&p_{1,3}(h\dddot{h}c)=-\dot{h}\ddot{h}c-\frac12h\ddot{h}\dot{c}+\frac12h\dot{h}\ddot{c}=-\frac32\dot{h}\ddot{h}c+\frac12d(h\dot{h}\ddot{h}),\\*
&p_{1,3}(\dddot{c})=0,\\*
&p_{1,3}(h\dddot{c})=-\frac12\dot{h}\ddot{c}+\frac12\ddot{h}\dot{c}=\frac12d(\dot{h}\ddot{h}),\\*
&p_{1,3}(h\dddot{c}c)=-\frac12\dot{h}\ddot{c}c+\dot{h}\dot{c}^2+\frac12\ddot{h}\dot{c}c=\dot{h}\dot{c}^2+\frac12d(\dot{h}\ddot{h}c),\\
&p_{1,3}(\dot{h}\ddot{h})=\dot{h}\ddot{h},\\*
&p_{1,3}(\dot{h}\ddot{h}c)=\dot{h}\ddot{h}c,\\*
&p_{1,3}(h\dot{h}\ddot{h})=h\dot{h}\ddot{h},\\*
&p_{1,3}(h\dot{h}\ddot{h}c)=h\dot{h}\ddot{h}c,\\*
&p_{1,3}(\dot{h}\ddot{c})=-\frac12\ddot{h}\dot{c}+\frac12\dot{h}\ddot{c}=-\frac12d(\dot{h}\ddot{h}),\\*
&p_{1,3}(\dot{h}\ddot{c}c)=\dot{h}\ddot{c}c=-\dot{h}\dot{c}^2+\dot{h}(\ddot{c}c+\dot{c}^2)=-\dot{h}\dot{c}^2\\*
&p_{1,3}(\dot{h}\dot{c}^2)=\dot{h}\dot{c}^2,\\*
&p_{1,3}(h\ddot{c}\dot{c})=-\frac12\dot{h}\dot{c}^2,\\
&p_{1,3}(h\dot{c}^3)=h\dot{c}^3=-h\ddot{c}c\dot{c}+h(\ddot{c}c+\dot{c}^2)\dot{c}=0.
\end{align*}
Hence $F_{0,3}$ is generated by $\dot{h}\dot{c}^2$.
On the other hand, we have $\dot{h}\dot{c}^2=-\dot{h}\ddot{c}c=d(\dot{h}\ddot{h}c)$.
Therefore, we have $H^*(\D^\infty\W_1)^3\cong\{0\}$.

Finally, we study $H^*(\D^\infty\W_1)^4$.
We have
\[
p_{1,4}=\id-\frac13\delta\circ\sigma+\frac1{15}\delta^2\circ\sigma^2-\frac1{90}\delta^3\circ\sigma^3,
\]
which is denoted by $p$ for simplicity.

We have
\begin{align*}
&p(h_{(4)})=0,\\*
&p(h_{(4)}c_{(0)})=-\frac15h_{(3)}c_{(1)}+\frac35h_{(2)}c_{(2)}-\frac15h_{(1)}c_{(3)},\\*
&p(h_{(0)}h_{(4)})=0,\\*
&p(h_{(0)}h_{(4)}c_{(0)})=-\frac15h_{(0)}h_{(3)}c_{(1)}+\frac95h_{(1)}h_{(2)}c_{(1)}+\frac35h_{(0)}h_{(2)}c_{(2)}-\frac15h_{(0)}h_{(1)}c_{(3)},\\*
&p(h_{(1)}h_{(3)})=0,\\*
&p(h_{(1)}h_{(3)}c_{(0)})=-h_{(1)}h_{(2)}c_{(1)},\\*
&p(h_{(0)}h_{(1)}h_{(3)})=0,\\*
&p(h_{(0)}h_{(1)}h_{(3)}c_{(0)})=-h_{(0)}h_{(1)}h_{(2)}c_{(1)},\\
&p(h_{(3)}c_{(1)})=\frac15h_{(3)}c_{(1)}-\frac35h_{(2)}c_{(2)}+\frac15h_{(1)}c_{(3)},\\*
&p(h_{(0)}h_{(3)}c_{(1)})=\frac15h_{(0)}h_{(3)}c_{(1)}-\frac45h_{(1)}h_{(2)}c_{(1)}-\frac35h_{(0)}h_{(2)}c_{(2)}+\frac15h_{(0)}h_{(1)}c_{(3)},\\*
&p(h_{(1)}h_{(2)}c_{(1)})=h_{(1)}h_{(2)}c_{(1)},\\*
&p(h_{(0)}h_{(1)}h_{(2)}c_{(1)})=h_{(0)}h_{(1)}h_{(2)}c_{(1)},\\
&p(h_{(2)}c_{(2)})=-\frac15h_{(3)}c_{(1)}+\frac35h_{(2)}c_{(2)}-\frac15h_{(1)}c_{(3)},\\*
&p(h_{(2)}c_{(2)}c_{(0)})=-\frac15h_{(3)}c_{(0)}c_{(1)}+\frac35h_{(2)}c_{(0)}c_{(2)}-\frac1{15}h_{(2)}c_{(1)}{}^2-\frac15h_{(1)}c_{(0)}c_{(3)}+\frac1{15}h_{(1)}c_{(1)}c_{(2)};\\*
&\quad\text{note that we have}\\*
&\quad h_{(2)}c_{(2)}c_{(0)}\begin{aligned}[t]
&=-h_{(2)}c_{(1)}{}^2\\*
&=d(h_{(1)}h_{(2)}c_{(1)})+h_{(1)}c_{(1)}c_{(2)}\\*
&=d(h_{(1)}h_{(2)}c_{(1)})-\frac13h_{(1)}c_{(0)}c_{(3)}\\*
&=d(h_{(1)}h_{(2)}c_{(1)})+\frac13d(h_{(1)}h_{(3)}c_{(0)}),
\end{aligned}\\*
&p(h_{(2)}c_{(1)}{}^2)=\frac23h_{(2)}c_{(1)}{}^2-\frac23h_{(1)}c_{(1)}c_{(2)},\\*
&p(h_{(0)}h_{(2)}c_{(2)})=-\frac15h_{(0)}h_{(3)}c_{(1)}+\frac35h_{(0)}h_{(2)}c_{(2)}-\frac15h_{(1)}h_{(2)}c_{(1)}-\frac15h_{(0)}h_{(1)}c_{(3)},\\*
&p(h_{(0)}h_{(2)}c_{(0)}c_{(2)})=\frac13h_{(0)}h_{(2)}c_{(0)}c_{(2)}-\frac13h_{(0)}h_{(2)}c_{(1)}{}^2-\frac13h_{(0)}h_{(1)}c_{(3)}c-\frac13h_{(0)}h_{(1)}c_{(1)}c_{(2)},\\*
&p(h_{(0)}h_{(2)}c_{(1)}{}^2)=\frac23h_{(0)}h_{(2)}c_{(1)}{}^2-\frac23h_{(0)}h_{(1)}c_{(1)}c_{(2)},\\
&p(h_{(1)}c_{(2)}c_{(1)})=-\frac13h_{(2)}c_{(1)}{}^2+\frac13h_{(1)}c_{(1)}c_{(2)},\\*
&p(h_{(0)}h_{(1)}c_{(1)}c_{(2)})=-\frac13h_{(0)}h_{(2)}c_{(1)}{}^2+\frac13h_{(0)}h_{(1)}c_{(1)}c_{(2)},\\
&p(h_{(0)}c_{(2)}{}^2)=\frac2{15}h_{(2)}c_{(1)}{}^2-\frac2{15}h_{(1)}c_{(1)}c_{(2)}+\frac35h_{(0)}c_{(2)}{}^2-\frac25h_{(0)}c_{(1)}c_{(3)};\\*
&\quad\text{note that we have}\\*
&\quad
h_{(0)}c_{(2)}{}^2=-h_{(0)}c_{(1)}c_{(1)}=d(h_{(0)}h_{(1)}c_{(1)}),\\*
&p(h_{(0)}c_{(1)}{}^2c_{(2)})=-\frac13h_{(1)}c_{(1)}{}^3=0,\\*
&p(h_{(0)}c_{(0)}c_{(2)}{}^2)=h_{(0)}c_{(0)}c_{(2)}{}^2=0,
\end{align*}
where we make use of Lemma~\ref{lem10.8}.

Therefore, cochains in $E_{0,4}$ are prepresented by linear combinations of the following elements:
\begin{align*}
&h_{(3)}c_{(1)}-3h_{(2)}c_{(2)}+h_{(1)}c_{(3)},\\*
&h_{(0)}h_{(3)}c_{(1)}-3h_{(0)}h_{(2)}c_{(2)}+h_{(0)}h_{(1)}c_{(3)},\\*
&h_{(1)}h_{(2)}c_{(1)},\\*
&h_{(0)}h_{(1)}h_{(2)}c_{(1)},\\*
&{}-3h_{(2)}c_{(0)}c_{(2)}+h_{(1)}c_{(0)}c_{(3)},\quad\text{(exact)}\\*
&h_{(2)}c_{(1)}{}^2-h_{(1)}c_{(1)}c_{(2)},\quad\text{(exact)}\\*
&-3h_{(0)}h_{(2)}c_{(0)}c_{(2)}+h_{(0)}h_{(1)}c_{(0)}c_{(3)},\quad\text{(exact)}\\*
&h_{(0)}h_{(2)}c_{(1)}{}^2-h_{(0)}h_{(1)}c_{(1)}c_{(2)},\quad\text{(exact)}\\*
&3h_{(0)}c_{(2)}{}^2-2h_{(0)}c_{(1)}c_{(3)}.\quad\text{(exact)}
\end{align*}
We can show that the above elements except the first four are exact.
Hence, only the second and the third one of the above elements can yield non-trivial classes.
The derivative of them are equal to
\begin{align*}
&h_{(0)}c_{(1)}c_{(3)}-3h_{(0)}c_{(2)}{}^2+3h_{(2)}c_{(0)}c_{(2)}+h_{(0)}c_{(1)}c_{(3)}-h_{(1)}c_{(0)}c_{(3)}\\*
&=2h_{(0)}c_{(1)}c_{(3)}+3h_{(2)}c_{(0)}c_{(2)}-h_{(1)}c_{(0)}c_{(3)},\\*
&h_{(1)}c_{(1)}c_{(2)}-h_{(2)}c_{(1)}{}^2.
\end{align*}
It is easy to see that they are linearly independent so that we have $F_{0,4}\cong\{0\}$.
\end{proof}

The eigenspace $F_{0,5}\subset H^*(\D^\infty\W_1)$ is non-trivial as follows.
Recall that if $E_{0,5}\subset\D^\infty\W_1$ denotes the $0$-eigenspace, then $F_{0,5}$ is the cohomology of $(E_{0,5},d)$.

\begin{theorem}
\label{thm10.7}
Let $V$ be the subspace of $F_{0,5}$ which consists of elements of type $(2,2)$.
Then, we have $V\cong\R$.
Indeed, $h_{(1)}h_{(2)}c_{(0)}c_{(2)}$ generates $V$.
\end{theorem}
\begin{proof}
Let $p=p_{1,5}$ be the projection from $(\D^\infty\W_1)^5$ to $F_{0,5}$.
Let $W$ the subspace of $E_{0,5}$ of type $(2,2)$ and $Z,B\subset W$ be the subspace of cocycles and coboundaries, respectively.
The following elements generate $W$:
\begin{align*}
&h_{(0)}h_{(1)}c_{(0)}c_{(4)},h_{(0)}h_{(1)}c_{(1)}c_{(3)},h_{(0)}h_{(1)}c_{(2)}{}^2,\\*
&h_{(0)}h_{(2)}c_{(0)}c_{(3)},h_{(0)}h_{(2)}c_{(1)}c_{(2)},\\*
&h_{(0)}h_{(3)}c_{(0)}c_{(2)},h_{(0)}h_{(3)}c_{(1)}{}^2,h_{(1)}h_{(2)}c_{(0)}c_{(2)},h_{(1)}h_{(2)}c_{(1)}{}^2.
\end{align*}
These elements are closed by Lemma~\ref{lem10.8}.
Among these elements,
\begin{align*}
&h_{(0)}h_{(1)}c_{(0)}c_{(4)},h_{(0)}h_{(1)}c_{(1)}c_{(3)},\\*
&h_{(0)}h_{(2)}c_{(0)}c_{(3)},\\*
&h_{(0)}h_{(3)}c_{(0)}c_{(2)},h_{(1)}h_{(2)}c_{(0)}c_{(2)}
\end{align*}
form a basis for $W$.
The image of these elements by $p$ are described as follows.
If $\omega\in W$, then $C(\omega)=(\lambda_1,\lambda_2,\ldots,\lambda_5)\in\R^5$ denotes the coordinate vector of $\omega$ with respect to the above basis, namely, we have $\omega=\lambda_1h_{(0)}h_{(1)}c_{(0)}c_{(4)}+\lambda_2h_{(0)}h_{(1)}c_{(1)}c_{(3)}+\lambda_3h_{(0)}h_{(2)}c_{(0)}c_{(3)}+\lambda_4h_{(0)}h_{(3)}c_{(0)}c_{(2)}+\lambda_5h_{(1)}h_{(2)}c_{(0)}c_{(2)}$.
Then, we have
\begin{align*}
C(p(h_{(0)}h_{(1)}c_{(0)}c_{(4)}))&=\left(-\tfrac1{14},-\tfrac{15}{14},-\tfrac9{14},\tfrac9{14},\tfrac9{14}\right)\\*
C(p(h_{(0)}h_{(1)}c_{(1)}c_{(3)}))&=\left(\tfrac5{28},\tfrac{33}{28},\tfrac3{28},-\tfrac3{28},-\tfrac3{28}\right)\\*
C(p(h_{(0)}h_{(2)}c_{(0)}c_{(3)}))&=\left(-\tfrac3{28},-\tfrac3{28},\tfrac{15}{28},-\tfrac{15}{28},-\tfrac{15}{28}\right)\\*
C(p(h_{(0)}h_{(3)}c_{(0)}c_{(2)}))&=\left(\tfrac1{14},\tfrac1{14},-\tfrac5{14},\tfrac5{14},-\tfrac9{14}\right)\\*
C(p(h_{(1)}h_{(2)}c_{(0)}c_{(2)}))&=\left(0,0,0,0,1\right).
\end{align*}
It follows that $\{(1,6,0,0,0),(0,1,1,-1,0),(0,0,0,0,1)\}$ is a basis for~$Z$.

Next we study the space $B$.
Elements of type $(3,1)$ and of order $5$ in $\D^\infty\W_1$ are linear combinations of the following elements:
%
\[
h_{(0)}h_{(1)}h_{(4)}c,h_{(0)}h_{(2)}h_{(3)}c,
h_{(0)}h_{(1)}h_{(3)}c_{(1)},
h_{(0)}h_{(1)}h_{(2)}c_{(2)}.
\]
The images of these elements by $p$ are given as follows:
\begin{align*}
&\frac17h_{(0)}h_{(1)}h_{(4)}c_{(0)}-\frac67h_{(0)}h_{(2)}h_{(3)}c_{(0)}-\frac3{14}h_{(0)}h_{(1)}h_{(3)}c_{(1)}+\frac9{14}h_{(0)}h_{(1)}h_{(2)}c_{(2)},\\*
&-\frac17h_{(0)}h_{(1)}h_{(4)}c_{(0)}+\frac67h_{(0)}h_{(2)}h_{(3)}c_{(0)}-\frac1{28}h_{(0)}h_{(1)}h_{(3)}c_{(0)}+\frac3{28}h_{(0)}h_{(1)}h_{(2)}c_{(2)},\\*
&\frac14h_{(0)}h_{(1)}h_{(3)}c_{(1)}-\frac34h_{(0)}h_{(1)}h_{(2)}c_{(2)},\\*
&\mathnormal{-}\frac14h_{(0)}h_{(1)}h_{(3)}c_{(1)}+\frac34h_{(0)}h_{(1)}h_{(2)}c_{(2)}.
\end{align*}
The span of these elements are equal to the one of $h_{(0)}h_{(1)}h_{(4)}c_{(0)}-6h_{(0)}h_{(2)}h_{(3)}c_{(0)}$ and $h_{(0)}h_{(1)}h_{(3)}c_{(1)}-3h_{(0)}h_{(1)}h_{(2)}c_{(2)}$.
By some honest calculations, we see that $\{(1,0,-6,6,0),(0,5,5,-5,-3)\}$ is a basis for $B$.
Therefore, we can choose\linebreak $\{(0,0,0,0,1)\}$ as a basis for $V=Z/B$.
\end{proof}

\begin{theorem}
\label{thm10.8}
The classes $\delta^1(\GV)\delta^4(\GV)$, $\delta^2(\GV)\delta^3(\GV)$ belong to $F_{0,5}$ and proportional to $h_{(1)}h_{(2)}c_{(0)}c_{(2)}$ which is non-trivial.
In particular, the product structure of $H^*(\D^\infty\W_1)$ is non-trivial.
\end{theorem}
\begin{proof}
We have
\begin{align*}
&\delta(\GV)=h_{(1)}c_{(0)}+h_{(0)}c_{(1)}=2h_{(1)}c_{(0)},\\*
&\delta^2(\GV)=2h_{(2)}c_{(0)}+2h_{(1)}c_{(1)},\\*
&\delta^3(\GV)=2h_{(3)}c_{(0)}+4h_{(2)}c_{(1)}+2h_{(1)}c_{(2)},\\*
&\delta^4(\GV)=2h_{(4)}c_{(0)}+6h_{(3)}c_{(1)}+6h_{(2)}c_{(2)}+2h_{(1)}c_{(3)}
\end{align*}
in $\D^\infty\W_1$.
It follows that we have $\delta^2(\GV)\delta^3(\GV)=4h_{(2)}h_{(1)}c_{(0)}c_{(2)}$ and that $\delta(\GV)\delta^4(\GV)=4h_{(1)}h_{(2)}c_{(0)}c_{(2)}$.
\end{proof}

\begin{remark}
The cochains $h_{(1)}h_{(2)}c_{(0)}c_{(2)}\in\D^k\W_1$ is not closed if $k\leq2$ and closed if $k\geq3$.
This resembles the fact that the Godbillon--Vey class is not defined for foliations of class $C^1$ \cite{Tsuboi}.
\end{remark}

\begin{remark}
We have $\delta(\GV)\delta^3(\GV)=0$ by Proposition~\ref{prop10.6}.
Hence we have $\delta^2(\GV)\delta^3(\GV)+\delta(\GV)\delta^4(\GV)=0$.
We also have $\sigma(\delta^2(\GV)\delta^3(\GV))=\delta(\GV)\delta^3(\GV)+3\delta^2(\GV)\delta^2(\GV)=0$.
We can show that $(\GV)\delta^5(\GV)$ is non-trivial as well, however, we omit the proof.
\end{remark}

We also have the following

\begin{proposition}
\label{prop10.2}
The class $h_{(0)}h_{(1)}h_{(2)}c_{(0)}c_{(2)}\in F_{0,5}$ is non-trivial.
\end{proposition}
\begin{proof}
Let $\omega=h_{(0)}h_{(1)}h_{(2)}c_{(0)}c_{(2)}$.
The cochain $\omega$ is closed by Lemma~\ref{lem10.8}.
We have $\sigma\omega=0$ so that $\omega$ represents an element of $F_{0,5}$.
As there are no non-trivial cochain of type $(4,1)$ of order $5$, $h_{(0)}h_{(1)}h_{(2)}c_{(0)}c_{(2)}$ is not exact.
\end{proof}

\begin{remark}
Because of the type, the class $h_{(0)}h_{(1)}h_{(2)}c_{(0)}c_{(2)}$ cannot be described in terms of $\delta^i(\GV)$ and $\delta^j(\FLK)$.
\end{remark}

\begin{remark}
The classes $h_{(1)}h_{(2)}c_{(0)}c_{(2)}$ and $h_{(0)}h_{(1)}h_{(2)}c_{(0)}c_{(2)}$ are formally variable by Corollary~\ref{cor9.16}.
However, we do \textit{not} know any example of deformations of which the classes $h_{(1)}h_{(2)}c_{(0)}c_{(2)}$ or $h_{(0)}h_{(1)}h_{(2)}c_{(0)}c_{(0)}c_{(2)}$ is non-trivial.
If they are always trivial, then there will exist some geometric constrains which yields a kind of the Bott vanishing or the Heitsch formulae.
Indeed, most of invariants known so far for $1$-dimensional dynamics are described by volumes, linear terms of mappings and Schwarzian derivative.
On the other hand, it is known that the classes $\GV$, $\delta(\GV)$ and $\FLK$ are described by the volume elements along the leaves, linear holonomy and the Schwarzian derivative of holonomy~\cite{asuke:tohoku}, \cite{asuke:Tohoku2017}, \cite{asuke:FLK}.
Theorem~\ref{thm10.8} implies that this is also the case for $h_{(1)}h_{(2)}c_{(0)}c_{(2)}$.
We expect that $h_{(0)}h_{(1)}h_{(2)}c_{(0)}c_{(2)}$ is still described by the above-mentioned invariants even if it is non-trivial for some examples.
\end{remark}

The non-triviality of the FLK class is derived from that of the derivative of the Bott class in~\cite{asuke:FLK}.
Shortly, integrating the FLK class $h_{(0)}h_{(1)}c_{(0)}$ along the direction of $h_{(0)}$, we obtain the derivative of the Bott class $h_{(1)}c_{(0)}$.
This construction is also valid for several other classes as follows.

First, we have the following

\begin{lemma}
Let $c$ be a cocycle in $\D^\infty\W_q$ and suppose that $h_1c$ is a non-trivial cocycle.
If $c$ is non-trivial in $H^*(\D^\infty\W_q)$, then $h_1c$ is also non-trivial in $H^*(\D^\infty\W_q)$.
\end{lemma}
\begin{proof}
Suppose that $h_1c=d\alpha$ as cocycles for some $\alpha\in\D^\infty\W_q$.
As $h_1c$ is non-trivial, $\alpha$ is also non-trivial.
We represent $\alpha$ as $\alpha=h_1\beta+\gamma$, where $\beta$ and $\gamma$ do not involve $h_1$.
Then, we have $h_1c=d\alpha=-h_1d\beta+c_1\beta+d\gamma$ so that we have $c=-d\beta$.
\end{proof}

Actually, a non-trivial class $c\in H^*(\D^\infty\W_q)$ yields a non-trivial class $h_1c$ under certain assumptions.
Let $\CF$ be a transversely holomorphic foliation of $M$ of complex codimension $q$.
Suppose that the complex normal bundle $Q(\CF)$ of $\CF$ is trivial and fix a trivialization, say $e$, of $Q(\CF)$.
Let $\widetilde{M}=S^1\times M$ and $\pi\colon\widetilde{M}\to M$ be the projection.
We equip $\widetilde{M}$ with the pull-back foliation $\pi^*\CF$.
If $\widehat\alpha\in\widehat{D}^r(\CF)$, then $\widehat\alpha$ naturally induces an element of $\widehat{D}^r(\widetilde{F})$ which is denoted by $\pi^*\widehat\alpha$.
For $c\in H^*(\D^\infty\W_q)$, we set $c(\CF,\widehat\alpha,e)=\widehat\chi_{\widehat\alpha}(c)$ in order to emphasize the trivialization $e$ of $Q(\CF)$.
Let $\widetilde\CF=\pi^*\CF$.
The normal bundle $Q(\widetilde\CF)$ is trivial.
If $m\in\Z$, we define trivialization $e_m$ of $Q(\widetilde\CF)$ as follows.
We regard $S^1$ as the unit circle in $\C$ and $t$ be the coordinate for $\C$.
Then, we set $\widetilde{e}_m=t^me$ using the product structure of $S^1\times M$.
Under these settings, we have the following

\begin{proposition}[\cite{asuke:FLK}, Proposition~3.1]
\label{prop10.14}
Let $c\in H^*(\D^\infty\W_q)$.
Assume that $c$ can be represented by a cocycle which does not involve $h_1,\ldots,h_q$ and that $h_1c$ is also a cocycle.
Suppose that $\delta c(\CF,\widehat\alpha,e)\in H^*(M;\C)$ is non-trivial.
If $m\neq0$, then either $\delta(h_1c)(\CF,\widehat\alpha,e)\in H^{*+1}(M;\C)$ or $\delta(h_1c)(\widetilde\CF,\pi^*\widehat\alpha,\widetilde{e}_m)\in H^{*+1}(\widetilde{M};\C)$ is non-trivial.
In particular $h_1c$ represents a variable class in $H^{*+1}(\D^\infty\W_q)$.
\end{proposition}
\begin{proof}
Let $\omega$ be the trivialization of $Q^*(\CF)$ dual to $e$ which is a $\C^q$-valued one-form.
It is known that there is a $M_q(\C)$-valued one-form, say $\eta$, such that $d\omega+\eta\wedge\omega=0$.
Actually, $\eta$ is a connection form of a Bott connection with respect to $e$ and we have $h_1=-\frac1{2\pi\sqrt{-1}}\tr\eta$ for $\CF$ and $e$.
We can choose $\omega_m=t^{-m}\omega$ as a trivialization of $Q^*(\widetilde\CF)$ which is the dual to $e_m$.
If we set $\eta_m=m\frac{dt}t+\eta$, then we have $d\omega_m+\eta_m\wedge\omega_m=0$.
We have $h_1=-\frac1{2\pi\sqrt{-1}}\tr\eta_m$ for $\widetilde\CF$ and $\widetilde{e}_m$.
As the deformation $\alpha$ is supported on $M$ and as $c$ does not involve $h_1,\ldots,h_q$ so that $c$ does not depend on the trivialization of $Q^*(\CF)$, we have
\[
h_1c(\widetilde\CF,\pi^*\widehat\alpha,\widetilde{e}_m)=\pi^*h_1c(\CF,\widehat\alpha,e)-qm\Vol_{S^1}\wedge\pi^*c(\CF,\widehat\alpha,e).
\]
As $\delta(c(\CF,\widehat\alpha,e))$ is assumed to be non-trivial and as $\Vol_{S^1}$ is independent of deformations, either $\delta(h_1c)(\widetilde\CF,\pi^*\widehat\alpha,\widetilde{e}_m)$ or $\delta(h_1c)(\CF,\widehat\alpha,e)$ is non-trivial.
\end{proof}

\begin{example}
If we set $c=\delta(h_{(0)}c_{(0)})=h_{(1)}c_{(0)}+h_{(0)}c_{(1)}=2h_{(1)}c_{(0)}$, then we obtain the twice of the FLK class $h_{(0)}h_{(1)}c_{(0)}$.
As an analogue, we set $c=\delta^2(h_{(0)}c_{(0)})=h_{(2)}c_{(0)}+2h_{(1)}c_{(1)}+h_{(0)}c_{(2)}=2(h_{(2)}c_{(0)}+h_{(1)}c_{(1)})$.
Then, we obtain the class $2h_{(0)}(h_{(2)}c_{(0)}+h_{(1)}c_{(1)})=2\delta(\FLK)$.
\end{example}
Some remarks and consequences of Proposition~\ref{prop10.14} will follow.

\begin{remark}
\begin{enumerate}
\item
The class $c$ is always formally variable.
Indeed, the class $c$ should involve some of $h_{i,(a)}$ with $a>0$, otherwise $c$ is trivial.
Hence we have $\ord(c)>0$ so that $\delta c\neq0$ by Theorem~\ref{structure}.
\item
It is well-known that we can calculate $h_1$ using $\Wedge^qQ^*(\CF)$.
Hence, if $c$ involves only $h_{1,(a)}$ and $c_{1,(b)}$, then we can make use of $\Wedge^qQ^*(\CF)$ instead of $Q^*(\CF)$.
It follows that we can avoid taking the trace so that the coefficients $qm$ in the above formula can be made into $m$.
This is the case for example if $c$ is derived from the Godbillon--Vey class and the FLK class.
\item
The above construction is not valid in the real category.
\item
So far as we know, the classes $c(\CF,\widehat\alpha,e)$ and $h_1c(\widetilde\CF,\pi^*\widehat\alpha,\widetilde{e}_m)$ are always proportional.
We do not know if there is a pair $\{c,h_1c\}$ which are linearly independent.
\item
Classes of the form $c$ are independent of the choice of trivializations.
\end{enumerate}
\end{remark}

Concerning Theroem~\ref{thm10.8} and Proposition~\ref{prop10.2}, we have the following

\begin{corollary}
Suppose that $h_{(1)}h_{(2)}c_{(0)}c_{(2)}(\CF,\widehat\alpha,e)\in H^6(M;\C)$ is non-trivial for some $\CF$, $\widehat{\alpha}$ and $e$.
Then, either $h_{(0)}h_{(1)}h_{(2)}c_{(0)}c_{(2)}(\CF,\widehat\alpha,e)\in H^7(M;\C)$ or\linebreak $h_{(0)}h_{(1)}h_{(2)}c_{(0)}c_{(2)}(\widetilde\CF,\pi^*\widehat\alpha,\widetilde{e}_m)\in H^7(\widetilde{M};\C)$ is non-trivial.
\end{corollary}

\end{document}